\documentclass[onefignum,onetabnum]{siamart190516}

\usepackage{lipsum}
\usepackage{amsfonts}
\usepackage{graphicx}
\usepackage{epstopdf}
\usepackage{algorithmic}
\usepackage{amsopn}

\usepackage{amsmath}
\usepackage{amssymb}
\usepackage{fancyhdr}
\usepackage{color}
\usepackage[caption=false]{subfig}
\usepackage{multirow}
\usepackage{rotating}
\usepackage{tabularx}
\usepackage{listings}
 \allowdisplaybreaks[4]

\newcommand{\tr}{^{\sf T}}
\newsiamremark{remark}{Remark}
\newsiamremark{hypothesis}{Hypothesis}
\crefname{hypothesis}{Hypothesis}{Hypotheses}
\newsiamthm{claim}{Claim}
\title{Equipping Barzilai-Borwein method with two dimensional quadratic termination property\thanks{This work was supported by the National Natural Science Foundation of China (Grant Nos. 11701137, 11631013, 12071108, 11671116, 11991020, 12021001) and Beijing Academy of Artificial Intelligence (BAAI).}}

\author{Yakui Huang\thanks{Institute of Mathematics, Hebei University of Technology, Tianjin 300401, China
  (\email{hyk@hebut.edu.cn}).}
\and Yu-Hong Dai\thanks{Corresponding author. LSEC, Academy of Mathematics and Systems Science, Chinese Academy of Sciences, Beijing 100190, China; School of Mathematical Sciences, University of Chinese Academy of Sciences, Beijing
	100049, China
  (\email{dyh@lsec.cc.ac.cn}, \url{http://lsec.cc.ac.cn/\string~dyh/}).}
\and Xin-Wei Liu\thanks{Institute of Mathematics, Hebei University of Technology, Tianjin 300401, China
  (\email{mathlxw@hebut.edu.cn}).}
}

\begin{document}

\maketitle

\begin{abstract}
A novel gradient stepsize is derived at the motivation of equipping the Barzilai-Borwein (BB) method with two dimensional quadratic termination property. A remarkable feature of the novel stepsize is that its computation only depends on the BB stepsizes in previous iterations and does not require any exact line search or the Hessian, and hence it can easily be extended for nonlinear optimization. By adaptively taking long BB steps and some short steps associated with the new stepsize, we develop an efficient gradient method for quadratic optimization and general unconstrained optimization and extend it to solve extreme eigenvalues problems. The proposed method is further extended for box-constrained optimization and singly linearly box-constrained optimization by incorporating gradient projection techniques. Numerical experiments demonstrate that the proposed method outperforms the most successful gradient methods in the literature.	
\end{abstract}

\begin{keywords}
  Barzilai-Borwein method, quadratic termination property, unconstrained optimization, extreme eigenvalues problem, box-constrained optimization, singly linearly box-constrained optimization
\end{keywords}

\begin{AMS}
  90C20, 90C25, 90C30
\end{AMS}

\pagestyle{myheadings}
\thispagestyle{plain}
\markboth{Y. HUANG, Y.-H. DAI and X.-W. LIU}
{EQUIPPING BB METHOD WITH TWO DIMENSIONAL QUADRATIC TERMINATION}

\section{Introduction}
\label{intro}


To minimize a smooth function $f(x): \mathbb{R}^n\to\mathbb{R}$, the gradient method updates iterates as follows
\begin{equation}\label{eqitr}
x_{k+1}=x_k-\alpha_kg_k,
\end{equation}
where $g_k=\nabla f(x_k)$ and $\alpha_k$ is the stepsize. Cauchy's (monotone) gradient method \cite{cauchy1847methode} (also known as the steepest descent (SD) method) chooses the stepsize, say $\alpha_k^{SD}$, by the exact line search. Although the SD method converges $Q$-linearly, it performs poorly in many problems due to  zigzag behaviors \cite{akaike1959successive,forsythe1968asymptotic}. By asking the stepsize to satisfy certain secant equations in the sense of least squares, Barzilai and Borwein \cite{Barzilai1988two} introduced the following long and short choices,
\begin{equation}\label{lbb}
\alpha_k^{BB1}=\arg\min_{\alpha\in\mathbb{R}}\|\alpha^{-1}s_{k-1}-y_{k-1}\|_2=\frac{s_{k-1} \tr s_{k-1}}{s_{k-1} \tr y_{k-1}}
\end{equation}
and
\begin{equation}\label{sbb}
\alpha_k^{BB2}=\arg\min_{\alpha\in\mathbb{R}}\|s_{k-1}-\alpha y_{k-1}\|_2=\frac{s_{k-1} \tr y_{k-1}}{y_{k-1} \tr y_{k-1}},
\end{equation}
where $s_{k-1}=x_k-x_{k-1}$ and $y_{k-1}=g_k-g_{k-1}$. Such genius stepsizes bring a surprisingly $R$-superlinear convergence in the two-dimensional strictly quadratic function \cite{Barzilai1988two}. For any dimensional strictly quadratic functions, the Barzilai-Borwein (BB) method is shown to be gobally convergent \cite{raydan1993} and the convergence is $R$-linear \cite{dai2002r}. See Li and Sun \cite{lisun2021} for an interesting improved $R$-linear convergence result of the BB method.
By cooperating with the nonmonotone line search by Grippo et al.
\cite{grippo1986nonmonotone}, Raydan \cite{raydan1997barzilai} first extended the BB method for unconstrained optimization, yielding a very efficient gradient method called GBB. Birgin et al. \cite{birgin2000nonmonotone} further extended the BB method to solve constrained optimization problems. Another significant work of the gradient method is due to Yuan \cite{yuan2006new,yuan2008step}, who suggested to calculate the stepsize such that, with the
previous and coming steps using SD stepsizes, the minimizer of a two dimensional strictly convex quadratic function is achieved in three iterations. By modifying Yuan's stepsize and alternating it with the SD stepsize in a suitable way, Dai and Yuan \cite{dai2005analysis} proposed the so-called Dai-Yuan (monotone) gradient method, which performs even better than the (nonmonotone) BB method.

In this paper, we shall introduce a new mechanism for the gradient method to achieve two dimensional quadratic termination. Interestingly, the aforementioned Yuan stepsize is a special example deduced from the mechanism. Furthermore, based on the mechanism, we derive a novel stepsize \eqref{snewa} such that the BB method equipped with the stepsize has the two dimensional quadratic termination property. A distinguished feature of the novel stepsize is that it is computed by BB stepsizes in two consecutive iterations and does not use any exact line search or the Hessian. Hence it can easily be extended to solve a wide class of unconstrained and constrained optimization problems.

By adaptively taking long BB steps and some short steps associated with the novel stepsize, we develop an efficient gradient method, the method \eqref{snbbmadp}, for quadratic optimization. Numerical experiments on minimizing quadratic functions show that the method \eqref{snbbmadp} performs much better than many successful gradient methods developed recently including BB1 \cite{Barzilai1988two}, Dai-Yuan (DY) \cite{dai2005analysis}, ABB \cite{zhou2006gradient}, ABBmin1 \cite{frassoldati2008new}, ABBmin2  \cite{frassoldati2008new} and SDC  \cite{de2014efficient}. The combination of the method \eqref{snbbmadp} and the Gripp-Lampariello-Lucidi (GLL) nonmonotone line search \cite{grippo1986nonmonotone} yields an efficient gradient method, Algorithm \ref{alunc}, for unconstrained optimization. Numerical experiments on unconstrained problems from the CUTEst collection \cite{gould2015cutest} show that Algorithm \ref{alunc} outperforms GBB \cite{raydan1997barzilai} and ABBmin \cite{Serafino2018,frassoldati2008new}.

Furthermore, with suitable modifications of the BB stepsizes and the use of the Dai-Fletcher nonmonotone line search \cite{dai2005projected}, the method \eqref{snbbmadp} is extended for extreme eigenvalues problems, yielding Algorithm \ref{aleig}. Numerical experiments demonstrate the advantage of Algorithm \ref{aleig} over EigUncABB \cite{jiang2014}. By incorporating gradient projection techniques and taking constraints into consideration, the method \eqref{snbbmadp} is further generalized for box-constrained optimization and singly linearly box-constrained (SLB) optimization, yielding Algorithm \ref{al1}. Numerical experiments on box-constrained problems from the CUTEst collection \cite{gould2015cutest} show that Algorithm \ref{al1} outperforms SPG   \cite{birgin2000nonmonotone,birgin2014spectral} and BoxVABBmin \cite{Crisci2019}. Meanwhile, numerical
experiments with random SLB problems and SLB problems arising in support vector machine \cite{cortes1995support} highly suggest the potential benefit of Algorithm \ref{al1} comparing the Dai-Fletcher method \cite{dai2006new}
and the EQ-VABBmin method \cite{Serena2020}.

The paper is organized as follows. In Section \ref{secnbb}, we introduce the new mechanism for the gradient method to achieve two dimensional quadratic termination. A novel stepsize \eqref{snewa} is derived such that the BB method equipped with the stepsize has such a property. Based on the novel stepsize, Section \ref{alqpunc} presents an efficient gradient method, the method \eqref{snbbmadp}, for quadratic optimization and an efficient gradient algorithm, Algorithm \ref{alunc}, for unconstrained optimization. Furthermore, Section \ref{secegen} provides an
efficient gradient algorithm, Algorithm \ref{aleig}, for solving extreme eigenvalues problems and Section \ref{alcon} provides an efficient gradient projection algorithm, Algorithm \ref{al1}, for box-constrained optimization and
singly linearly box-constrained optimization. Conclusion and discussion are made in Section \ref{secclu}.

\section{A mechanism and a novel stepsize for the gradient method}\label{secnbb}

In this section, we consider the unconstrained quadratic optimization
\begin{equation}\label{eqpro}
\min_{x\in \mathbb{R}^n}\ q(x):=\frac{1}{2}x \tr Ax-b \tr x,
\end{equation}
where $A\in \mathbb{R}^{n\times n}$ is symmetric positive definite and $b\in \mathbb{R}^n$.
At first, we shall provide a mechanism for the gradient method to achieve the two dimensional quadratic termination. Then we utilize the
mechanism to derive a novel stepsize such that the BB method can have the two dimensional quadratic termination by equipping with such stepsize.

\subsection{A mechanism for achieving two dimensional quadratic termination}
Our mechanism
starts from the following observation. Suppose that the gradient method is such that the gradient $g_{k+1}$ is
an eigenvector of the Hessian $A$; namely,
\begin{equation}\label{stepn1}
Ag_{k+1}=\lambda g_{k+1},
\end{equation}
where $\lambda$ is some eigenvalue of $A$. In the unconstrained quadratic case, it is easy to see
from \eqref{eqitr} and \eqref{stepn1} that
\begin{equation}\label{add1}
g_{k+2}=(I-\alpha_{k+1}A)g_{k+1}=(1-\alpha_{k+1}\lambda)g_{k+1},
\end{equation}
which means that $g_{k+2}$ is parallel to $g_{k+1}$ and $Ag_{k+2}=\lambda g_{k+2}$. Notice that for many stepsize formulae in the gradient
method, the stepsize $\alpha_{k+2}$ is the reciprocal of the Rayleigh quotient of the Hessian $A$ with respect to some vector in the form of $A^{\mu} g_{k+1}$ or $A^{\mu} g_{k+2}$, where $\mu$ is some real number; namely,
\begin{equation}\label{add2}
\alpha_{k+2}=\frac{(A^{\mu} g_{k+i}) \tr (A^{\mu} g_{k+i})}{(A^{\mu} g_{k+i}) \tr A(A^{\mu} g_{k+i})},
\quad\mbox{where $i=1$ or $2$}.
\end{equation}
Thus if neither $g_{k+1}$ nor $g_{k+2}$ vanish, we will have that $
\alpha_{k+2}=1/\lambda$ and hence $g_{k+3}=(1-\alpha_{k+2}\lambda)g_{k+2}=0$. Therefore
we must have that $g_{k+i}=0$ for some $1\le i\le 3$, yielding the finite termination.

As the relation \eqref{stepn1} is difficult to reach in the general case, we consider the dimensional of two; namely,
$n=2$. To this aim, let $\lambda_1$ and $\lambda_n$ be the smallest and largest eigenvalues of $A$, respectively, and
assume that $\psi$  is a real analytic function on $[\lambda_1,\lambda_n]$ and can be expressed by Laurent series $\psi(z)=\sum_{k=-\infty}^\infty c_kz^k$ where $c_k\in\mathbb{R}$ is such that $0<\sum_{k=-\infty}^\infty c_kz^k<+\infty$ for all $z\in[\lambda_1,\lambda_n]$. Then the relation \eqref{stepn1} implies that
\begin{equation}\label{stepinprod}
g_{\nu(k)} \tr \psi(A)g_{k+1}=\psi(\lambda) g_{\nu(k)} \tr g_{k+1},
\end{equation}
where $\nu(k)\in\{1,\ldots,k\}$. Furthermore, assume that $\psi_1$, $\psi_2$, $\psi_3$, $\psi_4$ are real analytic functions on $[\lambda_1,\lambda_n]$ with the same property of $\psi$, satisfying $\psi_1(z)\psi_2(z)=\psi_3(z)\psi_4(z)$ for $z\in\mathbb{R}$. It follows from \eqref{stepinprod} that
\begin{equation}\label{stepmuprod}
g_{\nu_1(k)} \tr \psi_1(A)g_{k+1}\cdot g_{\nu_2(k)} \tr \psi_2(A)g_{k+1}
=g_{\nu_1(k)} \tr \psi_3(A)g_{k+1}\cdot g_{\nu_2(k)} \tr \psi_4(A)g_{k+1},
\end{equation}
where  $\nu_1(k),\nu_2(k)\in\{1,\ldots,k\}$. Again, notice that in the unconstrained quadratic case,
the gradient method \eqref{eqitr} gives $g_{k+1}=(I-\alpha_kA)g_k$. The relation \eqref{stepmuprod}
provides a quadratic equation in stepsize $\alpha_k$ as follows.
\begin{align}\label{stepmuprod2}
&g_{\nu_1(k)} \tr \psi_1(A)(I-\alpha_kA)g_{k}\cdot g_{\nu_2(k)} \tr \psi_2(A)(I-\alpha_kA)g_{k}\nonumber\\
&=g_{\nu_1(k)} \tr \psi_3(A)(I-\alpha_kA)g_{k}\cdot g_{\nu_2(k)} \tr \psi_4(A)(I-\alpha_kA)g_{k}.
\end{align}
In the following, we shall show that if $n=2$, and if the gradient method preserves the invariance property under orthogonal transformations, we can derive the relation \eqref{stepn1} from the relation \eqref{stepmuprod2} and hence realizes a mechanism for achieving the two dimensional quadratic termination.

To proceed, we rewrite the quadratic equation \eqref{stepmuprod2} as follows,
\begin{equation}\label{stepqueq1}
\phi_1\alpha_k^2-\phi_2\alpha_k+\phi_3=0,
\end{equation}
where the coefficients are
\begin{equation*}
\begin{array}{rcl}
\phi_1&=&g_{\nu_1(k)} \tr \psi_1(A)Ag_{k}\cdot g_{\nu_2(k)} \tr \psi_2(A)Ag_{k}-
g_{\nu_1(k)} \tr \psi_3(A)Ag_{k}\cdot g_{\nu_2(k)} \tr \psi_4(A)Ag_{k},\\[2mm]
\phi_2&=&g_{\nu_1(k)} \tr \psi_1(A)g_{k}\cdot g_{\nu_2(k)} \tr \psi_2(A)Ag_{k}+
g_{\nu_1(k)} \tr \psi_1(A)Ag_{k}\cdot g_{\nu_2(k)} \tr \psi_2(A)g_{k}\nonumber\\[1mm]
&&-g_{\nu_1(k)} \tr \psi_3(A)g_{k}\cdot g_{\nu_2(k)} \tr \psi_4(A)Ag_{k}-
g_{\nu_1(k)} \tr \psi_3(A)Ag_{k}\cdot g_{\nu_2(k)} \tr \psi_4(A)g_{k}, \\[2mm]
\phi_3&=&g_{\nu_1(k)} \tr \psi_1(A)g_{k}\cdot g_{\nu_2(k)} \tr \psi_2(A)g_{k}-
g_{\nu_1(k)} \tr \psi_3(A)g_{k}\cdot g_{\nu_2(k)} \tr \psi_4(A)g_{k}.
\end{array}\end{equation*}
Then the two solutions of \eqref{stepmuprod2} or equivalently \eqref{stepqueq1} are
\begin{equation}\label{solutns}
\alpha_k=
\frac{\phi_2\pm\sqrt{\phi_2^2-4\phi_1\phi_3}}{2\phi_1}
=\frac{2}{\frac{\phi_2}{\phi_3}\mp\sqrt{\left(\frac{\phi_2}{\phi_3}\right)^2-4\frac{\phi_1}{\phi_3}}}.
\end{equation}

Now we are ready to give the following basic theorem for the two dimensional quadratic termination property.
\begin{theorem}\label{thnqft}
	Consider the gradient method \eqref{eqitr} with the invariance property under
	orthogonal transformations for the problem \eqref{eqpro} with $n=2$. Then if the stepsize
	solves the quadratic equation \eqref{stepmuprod2} and if $\alpha_{k+2}$ is of the form \eqref{add2}, we must have
	that $g_{k+i}=0$ for some $1\le i\le 3$.
\end{theorem}
\begin{proof}
	Due to the invariance property under orthogonal transformations and since $n=2$, we can
	assume without loss of generality that $A=\textrm{diag}\{1,\lambda\}$ with $\lambda>0$.
	
	Denote $g_k^{(i)}$, $g_{\nu_1(k)}^{(i)}$ and $g_{\nu_2(k)}^{(i)}$ to be the $i$-th components
	of $g_k$, $g_{\nu_1(k)}$ and $g_{\nu_2(k)}$, respectively, for $i=1,2$. Notice that $\psi_1(z)\psi_2(z)=\psi_3(z)\psi_4(z)$. Direct calculations show that the coefficients in \eqref{stepqueq1}
	are equal to  	
	\begin{equation*}\label{phi123}
	\phi_1=\lambda\varDelta,\quad \phi_2=(1+\lambda)\varDelta,\quad \phi_3=\varDelta,
	\end{equation*}
	where
	\begin{align*}
	\varDelta=&g_{\nu_1(k)}^{(1)}g_{\nu_2(k)}^{(2)}g_{k}^{(1)}g_{k}^{(2)}
	(\psi_1(1)\psi_2(\lambda)-\psi_3(1)\psi_4(\lambda))\\
	&+
	g_{\nu_1(k)}^{(2)}g_{\nu_2(k)}^{(1)}g_{k}^{(1)}g_{k}^{(2)}
	(\psi_1(\lambda)\psi_2(1)-\psi_3(\lambda)\psi_4(1)).
	\end{align*}
	Thus, by \eqref{solutns}, the two roots of \eqref{stepmuprod2} are
	\begin{equation*}\label{slts}
	\alpha_{k,1}=1
	~~\textrm{and}~~\alpha_{k,2}=\frac{1}{\lambda}.
	\end{equation*}
	
	If $\alpha_{k,1}$ is applied, by the update rule \eqref{eqitr}, we get that
	$g_{k+1}^{(1)}=(1-\alpha_{k,1})g_{k}^{(1)}=0$ and hence $Ag_{k+1}=\lambda g_{k+1}$.
	Thus if neither $g_{k+1}$ nor $g_{k+2}$ vanish, and if $\alpha_{k+2}$ is of the form \eqref{add2}, we will
	have that $\alpha_{k+2}=1/\lambda$ and hence $g_{k+3}=(1-\alpha_{k+2}\lambda)g_{k+2}=0$.
	Similar conclusion can be drawn for $\alpha_{k,2}$. So we must have that $g_{k+i}=0$ for some $1\le i\le 3$.
	This completes the proof.
\end{proof}

\begin{remark}
	By Theorem \ref{thnqft}, to achieve the two dimensional quadratic termination property for the stepsize $\alpha_k$ calculated from \eqref{stepmuprod2}, the choices of $\nu_1(k),\nu_2(k)$, $\psi_1,\psi_2$, $\psi_3$, and $\psi_4$ can arbitrarily be provided that $\psi_1(z)\psi_2(z)=\psi_3(z)\psi_4(z)$ for $z\in\mathbb{R}$. Consequently, in the two dimensional case, such a stepsize always yields finite termination of both the SD and BB methods, since the stepsizes in the methods are clearly of the form \eqref{add2}.
\end{remark}

The Dai-Yuan gradient method \cite{dai2005analysis} employs the following stepsize
\begin{equation}\label{syv}
\alpha_k^{DY}=\frac{2}{\frac{1}{\alpha_{k-1}^{SD}}+\frac{1}{\alpha_{k}^{SD}}+
	\sqrt{\left(\frac{1}{\alpha_{k-1}^{SD}}-\frac{1}{\alpha_{k}^{SD}}\right)^2+
		\frac{4\|g_k\|_2^2}{(\alpha_{k-1}^{SD}\|g_{k-1}\|_2)^2}}},
\end{equation}
which has the two dimensional quadratic termination property for the SD method. Here we notice that
the stepsize \eqref{syv} is a variant of the stepsize by Yuan \cite{yuan2006new}. Now we show that
$\alpha_k^{DY}$ is a special solution of \eqref{stepmuprod2}.

\begin{theorem}\label{recvdy}
	Suppose that  $\alpha_{k-1}=\alpha_{k-1}^{SD}$ and $I-\alpha_{k-1}A$ is invertible. Then the stepsize $\alpha_k^{DY}$ is a solution of the equation \eqref{stepmuprod2} corresponding to $\nu_1(k)=k-1$, $\nu_2(k)=k$, $\psi_1(A)=\psi_4(A)=(I-\alpha_{k-1}A)^{-1}$ and $\psi_2(A)=\psi_3(A)=I$.
\end{theorem}
\begin{proof}
	From \eqref{stepmuprod2} and the choices of $\nu_1(k)$, $\nu_2(k)$, $\psi_1,\psi_2$, $\psi_3$, and $\psi_4$, we have
	\begin{equation*}\label{deducedy}
	g_{k-1} \tr (I-\alpha_kA)g_{k-1}\cdot g_k \tr (g_k-\alpha_kAg_k)
	=g_{k-1} \tr (g_k-\alpha_kAg_k)\cdot g_k \tr (I-\alpha_kA)g_{k-1}.
	\end{equation*}
	Denote $\zeta_{k,j}=g_k \tr A^jg_k$ for $j=0,1$. It follows that
	\begin{equation}\label{deducedy2}
	(\zeta_{k-1,0}-\alpha_k\zeta_{k-1,1})(\zeta_{k,0}-\alpha_k\zeta_{k,1})
	=(g_{k-1} \tr g_k-g_{k-1} \tr Ag_k)^2.
	\end{equation}
	Since $\alpha_{k-1}=\alpha_{k-1}^{SD}$ and $g_k=(I-\alpha_{k-1}A)g_{k-1}$, we obtain $g_{k-1} \tr g_k=0$ and
	\begin{equation*}
	g_{k-1} \tr Ag_k=\frac{1}{\alpha_{k-1}^{SD}}(g_{k-1}-g_k) \tr g_k=-\frac{1}{\alpha_{k-1}^{SD}}\zeta_{k,0},
	\end{equation*}
	which together with \eqref{deducedy2} gives the equation \eqref{stepqueq1} with
	\begin{equation*}\label{deducedy3}
	\phi_1=\zeta_{k-1,1}\zeta_{k,1}-\frac{1}{(\alpha_{k-1}^{SD})^2}\zeta_{k,0}^2,~
	\phi_2=\zeta_{k,0}\zeta_{k-1,1}+\zeta_{k-1,0}\zeta_{k,1},~
	\phi_3=\zeta_{k-1,0}\zeta_{k,0}.
	\end{equation*}
	Consequently, we have that
	\begin{equation}\label{deducedyphi13}
	\frac{\phi_1}{\phi_3}=\frac{\zeta_{k-1,1}\zeta_{k,1}-\frac{1}{(\alpha_{k-1}^{SD})^2}\zeta_{k,0}^2}
	{\zeta_{k-1,0}\zeta_{k,0}}=\frac{1}{\alpha_{k-1}^{SD}\alpha_k^{SD}}-\frac{\|g_k\|_2^2}{\left(\alpha_{k-1}^{SD}\|g_{k-1}\|_2\right)^2}
	\end{equation}
	and
	\begin{equation}\label{deducedyphi23}
	\frac{\phi_2}{\phi_3}=\frac{\zeta_{k,0}\zeta_{k-1,1}+\zeta_{k-1,0}\zeta_{k,1}}
	{\zeta_{k-1,0}\zeta_{k,0}}=\frac{1}{\alpha_{k-1}^{SD}}+\frac{1}{\alpha_k^{SD}},
	\end{equation}
	where the last equality in \eqref{deducedyphi13} is due to the definition of $\zeta_{k,0}$. Thus we conclude from \eqref{solutns} that the stepsize $\alpha_k^{DY}$ is the smaller root of \eqref{stepmuprod2}. This completes the proof.
\end{proof}

\begin{remark}
	The stepsize $\alpha_k^{DY}$ has the two dimensional quadratic termination property only when applied for the SD method because it uses the relation $g_{k-1} \tr g_k=0$ to eliminate some terms of the equation \eqref{stepmuprod2}.
\end{remark}

\begin{remark}
For the family of gradient methods with
$$\alpha_k= \frac{g_{k-1} \tr \psi(A) g_{k-1}}{g_{k-1} \tr \psi(A)Ag_{k-1}},$$
Huang et al. \cite{hdlz2019} presented a stepsize, called $\tilde{\alpha}_k^H$, which enjoys the two dimensional quadratic termination property. Using the same way, under the assumption that $I-\alpha_{k-1}A$ and $I-\alpha_{k-2}A$ are invertible, we can deduce from Theorem \ref{recvdy} that for the same $r \in \mathbb{R}$, the stepsize $\tilde{\alpha}_k^H$ is a solution of the equation \eqref{stepmuprod2} corresponding to $\nu_1(k)=k-2$, $\nu_2(k)=k$, $\psi_1(A)=(I-\alpha_{k-2}A)^{-2}(I-\alpha_{k-1}A)^{-1}\psi^{2r}(A)$, $\psi_2(A)=\psi^{2(1-r)}(A)$, $\psi_3(A)=(I-\alpha_{k-2}A)^{-1}\psi(A)$, and $\psi_4(A)=(I-\alpha_{k-2}A)^{-1}(I-\alpha_{k-1}A)^{-1}\psi(A)$.
The drawback of this stepsize $\tilde{\alpha}_k^H$ is that the Hessian is involved in its calculation.
\end{remark}

As shown in Theorem \ref{thnqft}, for a two dimensional strictly convex quadratic function, the two roots of \eqref{stepmuprod2} are reciprocals of the largest and smallest eigenvalues of the Hessian $A$. This is also true, in the sense of limitation for the stepsize $\alpha_k^{DY}$  and the corresponding larger root when $n>2$, see \cite{huang2019asymptotic}. The reason why the smaller stepsize is preferable is that the larger one is generally not a good approximation of the reciprocal of the largest eigenvalue and may significantly increase the objective value, see \cite{huang2019asymptotic,huang2019gradient} for details.

\subsection{A novel stepsize equipped for the BB method}\label{secqbb}

Now we shall equip the BB method with a new stepsize via the quadratic equation \eqref{stepmuprod2} to achieve the two dimensional quadratic termination.

To this aim, suppose that $I-\alpha_{k-2}A$ and $I-\alpha_{k-1}A$ are invertible and let $\nu_1(k)=k-2$ and $\nu_2(k)=k-1$.  Furthermore, consider the following $\psi_i$ $(i=1,\ldots,4)$,
\begin{equation*}
\begin{array}{l}
\psi_1(A)=(I-\alpha_{k-2}A)^{-1}, \\  \psi_2(A)=(I-\alpha_{k-1}A)^{-1}, \\ \psi_3(A)=(I-\alpha_{k-2}A)^{-1}(I-\alpha_{k-1}A)^{-1}, \\  \psi_4(A)=I,
\end{array}
\end{equation*}
which satisfy $\psi_1(z)\psi_2(z)=\psi_3(z)\psi_4(z)$ for $z\in\mathbb{R}$. In this case, the equation
\eqref{stepmuprod2} becomes
\begin{align}\label{eqtodnbb}
&g_{k-2} \tr (I-\alpha_{k-2}A)^{-1}(I-\alpha_kA)g_{k}\cdot g_{k-1} \tr (I-\alpha_{k-1}A)^{-1}(I-\alpha_kA)g_{k} \nonumber\\
&=g_{k-2} \tr (I-\alpha_{k-2}A)^{-1}(I-\alpha_{k-1}A)^{-1}(I-\alpha_kA)g_{k}\cdot g_{k-1} \tr (I-\alpha_kA)g_{k}.
\end{align}
In order to get the formula of the novel stepsize, we assume for the moment that the equation \eqref{eqtodnbb} has a solution. This issue will be specified by Theorem \ref{thnbb1upbd}.

		\begin{theorem}\label{thnbb1}
			Suppose that $\alpha_{k-1}^{BB1}\neq\alpha_{k}^{BB1}$. Then the following stepsize $\alpha_k^{new}$  is a solution of the quadratic equation \eqref{eqtodnbb},
			\begin{equation}\label{snewa}
			\alpha_k^{new}
			=\frac{2}{\frac{\phi_2}{\phi_3}+\sqrt{\left(\frac{\phi_2}{\phi_3}\right)^2-4\,\frac{\phi_1}{\phi_3}}},
			\end{equation}
			where
			\begin{equation}\label{eqphinbb1}
			\frac{\phi_1}{\phi_3}
			=\frac{\alpha_{k-1}^{BB2}-\alpha_{k}^{BB2}}
			{\alpha_{k-1}^{BB2}\alpha_{k}^{BB2}(\alpha_{k-1}^{BB1}-\alpha_{k}^{BB1})}~~and~~
			\frac{\phi_2}{\phi_3}
			=\frac{\alpha_{k-1}^{BB1}\alpha_{k-1}^{BB2}-\alpha_{k}^{BB1}\alpha_{k}^{BB2}}
			{\alpha_{k-1}^{BB2}\alpha_{k}^{BB2}(\alpha_{k-1}^{BB1}-\alpha_{k}^{BB1})}.
			\end{equation}
		\end{theorem}
		\begin{proof}
			Denote $\zeta_{k,j}=g_k \tr A^jg_k$ for $j=0,1,2$ as before. By the update rule \eqref{eqitr}, we know that
			$g_{k+1}=(I-\alpha_kA)g_k$ holds for all $k\ge 1$ in the context of unconstrained quadratic optimization. By
			the assumption, we can get that
			\begin{align*}
			&\!\!\!\! g_{k-2} \tr (I-\alpha_{k-2}A)^{-1}(I-\alpha_kA)g_{k}\\
			&=g_{k-2} \tr (I-\alpha_{k}A)(I-\alpha_{k-1}A)g_{k-2}\\
			&=\zeta_{k-2,0}-\alpha_{k-1}\zeta_{k-2,1}-\alpha_{k}(\zeta_{k-2,1}-\alpha_{k-1}\zeta_{k-2,2})\\
			&=\zeta_{k-2,1}(\alpha_{k-1}^{BB1}-\alpha_{k-1})-\alpha_{k}\zeta_{k-2,2}(\alpha_{k-1}^{BB2}-\alpha_{k-1}),\\[2mm]
			&\!\!\!\!g_{k-1} \tr (I-\alpha_{k-1}A)^{-1}(I-\alpha_kA)g_{k}\\
			&= g_{k-1} \tr (I-\alpha_{k}A)g_{k-1} \\
			&=\zeta_{k-1,1}(\alpha_{k}^{BB1}-\alpha_{k}),\\[2mm]
			&\!\!\!\!g_{k-2} \tr (I-\alpha_{k-2}A)^{-1}(I-\alpha_{k-1}A)^{-1}(I-\alpha_kA)g_{k} \\
			&=g_{k-2} \tr (I-\alpha_{k}A)g_{k-2}\\
			&=\zeta_{k-2,1}(\alpha_{k-1}^{BB1}-\alpha_{k}),\\[2mm]
			&\!\!\!\!g_{k-1} \tr (I-\alpha_kA)g_{k}\\
			&= g_{k-1} \tr (I-\alpha_{k}A)(I-\alpha_{k-1}A)g_{k-1}\\
			&=\zeta_{k-1,1}(\alpha_{k}^{BB1}-\alpha_{k-1})-
			\alpha_{k}\zeta_{k-1,2}(\alpha_{k}^{BB2}-\alpha_{k-1}).
			\end{align*}
			The equation \eqref{eqtodnbb} can be written in the form \eqref{stepqueq1} with
			\begin{equation*}
			\begin{array}{rcl}
			\phi_1&=&\zeta_{k-2,2}\zeta_{k-1,1}(\alpha_{k-1}^{BB2}-\alpha_{k-1})-
			\zeta_{k-2,1}\zeta_{k-1,2}(\alpha_{k}^{BB2}-\alpha_{k-1})\\[1mm]
			&=&\zeta_{k-2,2}\zeta_{k-1,2}\big[\alpha_{k}^{BB2}(\alpha_{k-1}^{BB2}-\alpha_{k-1})-
			\alpha_{k-1}^{BB2}(\alpha_{k}^{BB2}-\alpha_{k-1})\big]\\[1mm]
			&=&\zeta_{k-2,2}\zeta_{k-1,2}\alpha_{k-1}(\alpha_{k-1}^{BB2}-\alpha_{k}^{BB2}), \\[3mm]
			\phi_2&=&\zeta_{k-2,2}\zeta_{k-1,1}\alpha_{k}^{BB1}(\alpha_{k-1}^{BB2}-\alpha_{k-1})
			+\zeta_{k-2,1}\zeta_{k-1,1}(\alpha_{k-1}^{BB1}-\alpha_{k-1})\\[1mm]
			&&-
			\zeta_{k-2,1}\zeta_{k-1,2}\alpha_{k-1}^{BB1}(\alpha_{k}^{BB2}-\alpha_{k-1})
			-\zeta_{k-2,1}\zeta_{k-1,1}(\alpha_{k}^{BB1}-\alpha_{k-1})\\[1mm]
			&=&\zeta_{k-2,2}\zeta_{k-1,2}\big[\alpha_{k}^{BB1}\alpha_{k}^{BB2}(\alpha_{k-1}^{BB2}-\alpha_{k-1})-
			\alpha_{k-1}^{BB1}\alpha_{k-1}^{BB2}(\alpha_{k}^{BB2}-\alpha_{k-1})\\[1mm]
			&&+\alpha_{k-1}^{BB2}\alpha_{k}^{BB2}(\alpha_{k-1}^{BB1}-\alpha_{k}^{BB1})\big]\\[1mm]
			&=&\zeta_{k-2,2}\zeta_{k-1,2}\alpha_{k-1}(\alpha_{k-1}^{BB1}\alpha_{k-1}^{BB2}-\alpha_{k}^{BB1}\alpha_{k}^{BB2}),\\[3mm]
			\phi_3&=&\zeta_{k-2,1}\zeta_{k-1,1}\alpha_{k-1}(\alpha_{k-1}^{BB1}-\alpha_{k}^{BB1}).
			\end{array}
			\end{equation*}
			Therefore we obtain
			\begin{equation*}
			\begin{array}{rcl}
			\dfrac{\phi_1}{\phi_3}&=&\dfrac{\zeta_{k-2,2}\zeta_{k-1,2}\alpha_{k-1}(\alpha_{k-1}^{BB2}-\alpha_{k}^{BB2})}
			{\zeta_{k-2,1}\zeta_{k-1,1}\alpha_{k-1}(\alpha_{k-1}^{BB1}-\alpha_{k}^{BB1})}\nonumber\\[1mm]
			&=&\dfrac{\alpha_{k-1}^{BB2}-\alpha_{k}^{BB2}}
			{\alpha_{k-1}^{BB2}\alpha_{k}^{BB2}(\alpha_{k-1}^{BB1}-\alpha_{k}^{BB1})}, \\[5mm]
			\dfrac{\phi_2}{\phi_3}&=&\dfrac{\zeta_{k-2,2}\zeta_{k-1,2}\alpha_{k-1}(\alpha_{k-1}^{BB1}\alpha_{k-1}^{BB2}-\alpha_{k}^{BB1}\alpha_{k}^{BB2})}
			{\zeta_{k-2,1}\zeta_{k-1,1}\alpha_{k-1}(\alpha_{k-1}^{BB1}-\alpha_{k}^{BB1})}\nonumber\\[1mm]
			&=&\dfrac{\alpha_{k-1}^{BB1}\alpha_{k-1}^{BB2}-\alpha_{k}^{BB1}\alpha_{k}^{BB2}}
			{\alpha_{k-1}^{BB2}\alpha_{k}^{BB2}(\alpha_{k-1}^{BB1}-\alpha_{k}^{BB1})}.
			\end{array} \end{equation*}
			This completes the proof by noting that $\alpha_k^{new}$ is the smaller solution of \eqref{stepqueq1}.
		\end{proof}

\begin{remark}
	Although the derivation of $\alpha_k^{new}$ is based on unconstrained quadratic optimization, it can be applied for general nonlinear optimization as well because the formula is only related to BB stepsizes in previous two iterations.
\end{remark}

		\begin{remark}
			Interestingly enough, by the proof of Theorem \ref{thnbb1}, the stepsize $\alpha_k^{new}$ is independent of the gradient method. In other words, it satisfies the equation \eqref{eqtodnbb} for any gradient method such as the SD, BB, alternate BB \cite{dai2005projected} and cyclic BB \cite{dai2006cyclic} methods.
		\end{remark}


Since the equation \eqref{eqtodnbb} is a special case of \eqref{stepmuprod2}, we know by Theorems \ref{thnqft} and \ref{thnbb1} that the stepsize $\alpha_k^{new}$ has the desired two dimensional quadratic termination property for both the BB1 and BB2 methods. For the purpose of  numerical verification, we applied the BB1 and BB2 methods with $\alpha_3$ replaced with $\alpha_3^{new}$ for the unconstrained quadratic minimization problem \eqref{eqpro} with
\begin{equation}\label{twoquad}
A=\textrm{diag}\{1,\,\lambda\} \quad \mbox{and} \quad b=0.
\end{equation}
The algorithm was run for five iterations using ten random starting points. Table \ref{tb2ft} presents averaged values of $\|g_6\|_2$ and $f(x_6)$. It can be observed that for different $\lambda$, the values of $\|g_6\|_2$ and $f(x_6)$ obtained by the BB1 and BB2 methods with $\alpha_3^{new}$ are numerically very close to zero, whereas those values obtained by the unmodified BB1 method are far away from zero.

\begin{table}[ht!b]
	\setlength{\tabcolsep}{1ex}
	\linespread{1.5}
	{\footnotesize
		\caption{Averaged results on problem \eqref{twoquad} with different condition numbers.}\label{tb2ft}
		\begin{center}
			\begin{tabular}{|c|c|c|c|c|c|c|}
				\hline
				\multicolumn{1}{|c|}{\multirow{2}{*}{$\lambda$}}
				&\multicolumn{2}{c|}{\multirow{1}{*}{BB1}}
				&\multicolumn{2}{c|}{\multirow{1}{*}{BB1 with $\alpha_3^{new}$}}
				&\multicolumn{2}{c|}{\multirow{1}{*}{BB2 with $\alpha_3^{new}$}}\\
				\cline{2-7}
				&\multicolumn{1}{c|}{$\|g_6\|_2$} &\multicolumn{1}{c|}{$f(x_6)$}
				&\multicolumn{1}{c|}{$\|g_6\|_2$} &\multicolumn{1}{c|}{$f(x_6)$}
				&\multicolumn{1}{c|}{$\|g_6\|_2$} &\multicolumn{1}{c|}{$f(x_6)$}\\
				\hline
				10 &6.9873e-01  &7.5641e-02  &9.3863e-18  &1.7612e-32  &7.5042e-20  &5.1740e-33\\
				\hline
				100 &6.3834e+00  &1.8293e+00  &1.3555e-17  &1.1388e-31  &8.6044e-17  &3.2604e-31\\
				\hline
				1000 &1.5874e+00  &6.6377e-03  &8.8296e-16  &3.2198e-30  &3.7438e-28  &9.5183e-31\\
				\hline
				10000 &2.9710e+01  &2.1038e-01  &8.3267e-17  &3.2828e-30  &2.0988e-31  &8.0889e-30\\
				\hline
			\end{tabular}
		\end{center}
	}
\end{table}

The next theorem indicates that the quadratic equation \eqref{eqtodnbb} always has a solution. It also provides upper and lower bounds for $\alpha_k^{new}$. A simple way of computing $\phi_2/\phi_3$ is given in the proof of the theorem. Here we do not consider the trivial case $\phi_3=0$.

\begin{theorem}\label{thnbb1upbd}
	The stepsize $\alpha_k^{new}$ in Theorem \ref{thnbb1} is well defined. Moreover, if $\phi_1/\phi_3\geq0$ and $\phi_2\neq0$, we have that
	\begin{equation}\label{nbb1lbd}
	\phi_3/\phi_2\leq\alpha_k^{new}\leq\min\{\alpha_{k}^{BB2},\alpha_{k-1}^{BB2}\};
	\end{equation}
if $\phi_1/\phi_3<0$ and $\phi_2\neq0$, we have that
		\begin{equation}\label{nbb1lbd2}
		\max\{\alpha_{k}^{BB2},\alpha_{k-1}^{BB2}\}\leq\alpha_k^{new}\leq\left|\phi_3/\phi_2\right|.
		\end{equation}
\end{theorem}
\begin{proof}
	If $\phi_1/\phi_3<0$, we see from \eqref{snewa} that $\alpha_k^{new}$ is well defined. Now we consider the case that $\phi_1/\phi_3\geq0$. Notice that $\phi_2/\phi_3$ can be rewritten as
	\begin{align}\label{nbbphi232}
	\frac{\phi_2}{\phi_3}
	&=\frac{\alpha_{k-1}^{BB2}(\alpha_{k-1}^{BB1}-\alpha_{k}^{BB1})+\alpha_{k}^{BB1}(\alpha_{k-1}^{BB2}-\alpha_{k}^{BB2})}
	{\alpha_{k-1}^{BB2}\alpha_{k}^{BB2}(\alpha_{k-1}^{BB1}-\alpha_{k}^{BB1})}\nonumber\\
	&=\frac{1}{\alpha_{k}^{BB2}}+\frac{\phi_1}{\phi_3}\alpha_{k}^{BB1}.
	\end{align}
	Since $\alpha_{k}^{BB1}\ge \alpha_{k}^{BB2}$, by \eqref{nbbphi232} and \eqref{eqphinbb1}, we obtain
	\begin{align*}\label{nbbdlt}
	\left(\frac{\phi_2}{\phi_3}\right)^2
	&\geq\left(\frac{\phi_2}{\phi_3}\right)^2-4\,\frac{\phi_1}{\phi_3}\nonumber\\
	&=
	\frac{1}{(\alpha_{k}^{BB2})^2}+\left(\frac{\phi_1}{\phi_3}\right)^2(\alpha_{k}^{BB1})^2+
	\,2\frac{\phi_1}{\phi_3}\frac{\alpha_{k}^{BB1}}{\alpha_{k}^{BB2}}-4\,
	\frac{\phi_1}{\phi_3}\nonumber\\
	&\geq\frac{1}{(\alpha_{k}^{BB2})^2}+\left(\frac{\phi_1}{\phi_3}\right)^2(\alpha_{k}^{BB1})^2-
	2\,\frac{\phi_1}{\phi_3}\frac{\alpha_{k}^{BB1}}{\alpha_{k}^{BB2}}\nonumber\\
	&=\left(\frac{1}{\alpha_{k}^{BB2}}-\frac{\phi_1}{\phi_3}\alpha_{k}^{BB1}\right)^2,
	\end{align*}
	which with \eqref{snewa} indicates that $\alpha_k^{new}$ is well defined and
	\begin{equation}\label{snew1upbd1}
	\frac{\phi_3}{\phi_2}
	\leq\alpha_k^{new}\leq\min\left\{\alpha_{k}^{BB2},\,\frac{1}{\alpha_{k}^{BB1}}\frac{\phi_3}{\phi_1}\right\}
	\leq\alpha_{k}^{BB2}.
	\end{equation}
	Similarly to \eqref{nbbphi232}, we can get that
	\begin{equation}\label{nbbphi233}
	\frac{\phi_2}{\phi_3}=\frac{1}{\alpha_{k-1}^{BB2}}+\frac{\phi_1}{\phi_3}\alpha_{k-1}^{BB1}.
	\end{equation}
	Hence, we obtain
	\begin{equation}\label{snew1upbd3}
	\frac{\phi_3}{\phi_2}\leq\alpha_k^{new}\leq\min\left\{\alpha_{k-1}^{BB2},\,\frac{1}{\alpha_{k-1}^{BB1}}\frac{\phi_3}{\phi_1}\right\}
	\leq\alpha_{k-1}^{BB2}.
	\end{equation}
	Combining \eqref{snew1upbd1} and \eqref{snew1upbd3} yields the desired inequalities \eqref{nbb1lbd}. This completes the proof of \eqref{nbb1lbd}.
	
			To prove \eqref{nbb1lbd2}, by $\phi_1/\phi_3<0$ and \eqref{nbbphi232}, we have that
			\begin{align*}
			\left(\frac{\phi_2}{\phi_3}\right)^2
			&\leq\left(\frac{\phi_2}{\phi_3}\right)^2-4\,\frac{\phi_1}{\phi_3}\leq\left(\frac{1}{\alpha_{k}^{BB2}}-\frac{\phi_1}{\phi_3}\alpha_{k}^{BB1}\right)^2,
			\end{align*}
			which with the fact $\alpha_{k}^{BB1}\geq0$ implies that $\alpha_{k}^{BB2}\leq\alpha_k^{new}\leq\left|\phi_3/\phi_2\right|$.
			Similarly, by \eqref{nbbphi233}, we can obtain $\alpha_{k-1}^{BB2}\leq\alpha_k^{new}\leq\left|\phi_3/\phi_2\right|$.
			Then the relation \eqref{nbb1lbd2} follows immediately.
\end{proof}

As will be seen in the next section, the above theorem stimulates us to provide a new choice for the short stepsize in the algorithmic design of the gradient method. Finally, we give an asymptotic spectral property of $\alpha_k^{new}$ within the SD method.
\begin{theorem}\label{thnbbspectral}
When applying the SD method to $n$-dimensional unconstrained quadratic problem \eqref{eqpro}, we have that $\lim_{k\rightarrow\infty}\alpha_k^{new}=1/\lambda_n$.
\end{theorem}
\begin{proof}
By Theorem 1 of \cite{huang2019asymptotic}, when applying the SD method to problem \eqref{eqpro}, we obtain
	\begin{equation*}
	\begin{array}{rcl}
	\lim\limits_{k\rightarrow\infty}\alpha_{2k}^{BB1}&=&\lim\limits_{k\rightarrow\infty}\alpha_{2k-1}^{SD}=
	\dfrac{1+c^2}{\lambda_1(\kappa+c^2)}, \\[2mm]
	\lim\limits_{k\rightarrow\infty}\alpha_{2k+1}^{BB1}&=&\lim\limits_{k\rightarrow\infty}\alpha_{2k}^{SD}=
	\dfrac{1+c^2}{\lambda_1(1+c^2\kappa)}, \\[2mm]
	\lim\limits_{k\rightarrow\infty}\alpha_{2k}^{BB2}&=&
	\dfrac{\kappa+c^2}{\lambda_1(\kappa^2+c^2)}, \\[2mm]
	\lim\limits_{k\rightarrow\infty}\alpha_{2k+1}^{BB2}&=&
	\dfrac{1+c^2\kappa}{\lambda_1(1+c^2\kappa^2)}, \\
	\end{array} \end{equation*}
	where $\kappa=\lambda_n/\lambda_1$ is the condition number of $A$ and $c$ is a constant. It follows from  \eqref{snewa} and \eqref{eqphinbb1} that
	\begin{equation*}
	\begin{array}{rcl}
	\lim\limits_{k\rightarrow\infty}\dfrac{\phi_1}{\phi_3}&=&\lambda_1\lambda_n, \\[2mm]
	\lim\limits_{k\rightarrow\infty}\dfrac{\phi_2}{\phi_3}&=&\lambda_1+\lambda_n, \\
	\end{array} \end{equation*}
	which yield $\lim_{k\rightarrow\infty}\alpha_k^{new}=1/\lambda_n$. This completes our proof.
\end{proof}

\section{A new gradient method for unconstrained optimization}\label{alqpunc}
In this section, by making use of the new stepsize $\alpha_k^{new}$, we develop an efficient gradient method for solving unconstrained optimization problems.

\subsection{Quadratic optimization}
To begin with, we consider the quadratic optimization \eqref{eqpro} again, which is often used for constructing and analyzing optimization methods and plays an important role in nonlinear optimization.

Extensive studies show that using the long BB stepsize $\alpha_{k}^{BB1}$ (since $\alpha_{k}^{BB1}\geq\alpha_{k}^{BB2}$) and some short stepsize in an alternate or adaptive manner is numerically better than the original BB method, see for example  \cite{Serena2020,dai2003alternate,dai2005projected,dai2006cyclic,frassoldati2008new,friedlander1998gradient,hdlz2019,raydan2002relaxed,zhou2006gradient}.
If $\phi_1/\phi_3\geq0$ and $\phi_2\neq0$, we know from Theorem \ref{thnbb1upbd} that the new stepsize
$\alpha_k^{new}$ is shorter than the two short stepsizes $\alpha_{k}^{BB2}$ and $\alpha_{k-1}^{BB2}$.
On the other hand, if $\phi_1/\phi_3<0$ and $\phi_2\neq0$, the stepsize $\alpha_k^{new}$ is longer than both
$\alpha_{k}^{BB2}$ and $\alpha_{k-1}^{BB2}$.  Combining the two cases, we shall replace $\alpha_k^{new}$ by $\min\{\alpha_{k-1}^{BB2},\alpha_{k}^{BB2},\alpha_k^{new}\}$, which is the shortest stepsize among $\alpha_{k-1}^{BB2}$, $\alpha_{k}^{BB2}$ and $\alpha_k^{new}$, in the algorithmic design of new gradient methods.

Motivated by the success of adaptive schemes \cite{Bonettini2008,zhou2006gradient}, we now suggest the  gradient method \eqref{eqitr} with the stepsizes $\{\alpha_k: k\ge 3\}$ given by
\begin{equation}\label{snbbmadp}
\alpha_k=\left\{
\begin{array}{ll}
\min\{\alpha_{k-1}^{BB2},\alpha_{k}^{BB2},\alpha_k^{new}\}, & \hbox{if $\alpha_k^{BB2}/\alpha_k^{BB1}<\tau_k$;} \\
\alpha_k^{BB1}, & \hbox{otherwise,}
\end{array}
\right.
\end{equation}
where $\tau_k>0$. In addition, we take $\alpha_1=\alpha_1^{SD}$ and $\alpha_2=\alpha_2^{BB1}$ for quadratic optimization. The simplest way to update $\tau_k$ in \eqref{snbbmadp} is setting it to some constant $\tau\in(0,1)$
for all $k$ (see \cite{zhou2006gradient}). For this fixed scheme, the performance of the method \eqref{snbbmadp} may heavily depend on the value of $\tau$. Another strategy is to update $\tau_k$ dynamically by
\begin{equation} \label{tauk}
\tau_{k+1}=\left\{
\begin{array}{ll}
\tau_k/\gamma, & \hbox{ if $\alpha_k^{BB2}/\alpha_k^{BB1}<\tau_k$;} \\
\tau_k\gamma, & \hbox{ otherwise,}
\end{array}
\right.
\end{equation}
for some $\gamma>1$, see \cite{Bonettini2008} for example. In what follows, we will present an intuitive comparison between the fixed and dynamic schemes.

We tested the new method  \eqref{snbbmadp} on some randomly generated quadratic problems \cite{yuan2006new}. The objective function is given by
\begin{equation}\label{quad-test1}
f(x)=(x-x^*) \tr \textrm{diag}\{v_1,\ldots,v_n\}(x-x^*),
\end{equation}
where $x^*$ was randomly generated with components in $[-10,10]$,
$v_1=1$, $v_n=\kappa$, and $v_j$, $j=2,\ldots,n-1$, were generated between $1$ and $\kappa$ by the \emph{rand} function in Matlab. The null vector was employed as the starting point.  We terminated the method if either the number of iteration exceeds $20000$ or
$\|g_k\|_2\leq\epsilon\|g_1\|_2$,
where $\epsilon$ is a given tolerance.  For each problem, we tested three different values of tolerances $\epsilon=10^{-6}$, $10^{-9}$, $10^{-12}$ and condition numbers $\kappa=10^4$, $10^5$, $10^6$. We randomly generated $10$ instances of the problem for each value of $\kappa$ or $\epsilon$.

\begin{table}[htp!b]  	
	\setlength{\tabcolsep}{0.7ex}                                                          	
	\caption{The average number of iterations required by the method  \eqref{snbbmadp} with fixed and dynamic schemes on  quadratic problem \eqref{quad-test1}.}\label{tbnbbrand1} 	
	\footnotesize
	\centering      	
	\begin{tabular}{ccccccccccccc}                                                   		
		\hline                                                               		
		\multicolumn{1}{c}{\multirow{2}{*}{$n$}}		
		&\multicolumn{1}{c}{\multirow{2}{*}{$\epsilon$}}		
		&\multicolumn{5}{c}{\multirow{1}{*}{Fixed scheme ($\tau$)}} 		
		&\multicolumn{1}{c}{}		
		&\multicolumn{5}{c}{\multirow{1}{*}{Dynamic scheme ($\tau_1$)}} \\ 		
		\cline{3-7} \cline{9-13}   		
		& & 0.1  & 0.2  & 0.5  & 0.7 & 0.9  & & 0.1  & 0.2  & 0.5  & 0.7 & 0.9 \\
		
		\hline
		
		
		\multicolumn{1}{c}{\multirow{3}{*}{1000}}
		&1e-06 & 211.4  & 202.2  & 198.7  & 194.8  & 194.4  &	
& 203.6  & 194.3  & 198.1  & 202.4  & 195.7  \\
		&1e-09 & 616.2  & 688.7  & 767.0  & 888.0  &1097.5  &
& 563.3  & 548.7  & 598.5  & 566.4  & 576.9  \\
		&1e-12 & 898.0  & 954.0  &1100.0  &1127.8  &1824.8  &
& 811.9  & 813.6  & 845.7  & 830.9  & 831.1  \\
		
		\hline
		
		\multicolumn{1}{c}{\multirow{3}{*}{10000}}                                                 	
&1e-06 & 263.2  & 261.3  & 232.8  & 233.9  & 244.9  &	& 251.1  & 245.6  & 250.1  & 250.0  & 243.3  \\
&1e-09 &1223.9  &1188.0  &1520.6  &1537.7  &1925.6  & &1092.3  &1206.6  &1196.4  &1240.6  &1229.3  \\
&1e-12 &1954.9  &1793.6  &2085.1  &2198.6  &3230.3  & &1949.5  &1932.3  &2034.2  &2055.1  &2055.2  \\

\hline 	
		           	
	\end{tabular}                                                               	
\end{table}

For the fixed scheme, the parameter $\tau$ varies from $0.1$, $0.2$, $0.5$, $0.7$, $0.9$. The dynamic scheme also uses these values for $\tau_1$ and takes the value $\gamma=1.02$. Table \ref{tbnbbrand1} presents the average numbers of iterations of the two schemes. Clearly, when $n=1000$, the dynamic scheme outperforms the fixed scheme for most values of $\tau_1$. When $n=10000$, the dynamic scheme is better than or comparable to the fixed one. Moreover, performance of the dynamic scheme is less dependent on the value of $\tau_1$. Hence, in what follows, we will concentrate on the dynamic scheme.

For the quadratic optimization \eqref{eqpro}, the $R$-linear global convergence of the new method \eqref{snbbmadp} can easily be  established by using the property in \cite{dai2003alternate}. See the proof of Theorem~3 in \cite{dhl2018} for example.


To further illustrate the efficiency of the new method \eqref{snbbmadp} for quadratic optimization, we compared it with the following gradient methods:
\begin{itemize}
	\item[(i)] BB1 \cite{Barzilai1988two}: the original BB method using $\alpha_k^{BB1}$;
	
	\item[(ii)] DY \cite{dai2005analysis}: the Dai-Yuan monotone gradient method;
	
	\item[(iii)] ABB \cite{zhou2006gradient}: the adaptive BB method;
	
	\item[(iv)] ABBmin1 \cite{frassoldati2008new}: a gradient method adaptively uses $\alpha_k^{BB1}$ and  $\min\{\alpha_j^{BB2}:~j=\max\{1,k-m\},\dots,k\}$;
	
	\item[(v)] ABBmin2 \cite{frassoldati2008new}: a gradient method adaptively uses $\alpha_k^{BB1}$ and a short stepsize in \cite{frassoldati2008new};
	
	\item[(vi)] SDC \cite{de2014efficient}: a gradient method takes $h$ SD iterates followed by $s$ steps with the same $\alpha_k^{DY}$.
	
\end{itemize}

Firstly, we tested the methods on the problem \eqref{quad-test1}. We employed the same settings for initial points, condition numbers and tolerances as before. Five different distributions of the diagonal elements $v_j$, $j=2,\ldots,n-1$, were generated (see Table \ref{tbspe}). The problem dimension was set to $n=10000$. For the SDC method, the parameter pair $(h,s)$ was set to $(8,6)$, which is more efficient than other choices. As suggested in \cite{frassoldati2008new}, $\tau=0.15$, $0.8$, $0.9$ were used for the ABB, ABBmin1, ABBmin2 methods, respectively, whereas $m=9$ was employed for the ABBmin1 method.

Figure \ref{frandp} presents performance profiles \cite{dolan2002} obtained by the new method \eqref{snbbmadp} with $\tau_1=0.2$, $\gamma=1.02$ and other methods using the average number of iterations as the metric.
For each method, the vertical axis of the figure shows the percentage of problems the method solves within the factor $\rho$ of the minimum value of the metric.
It can be seen that the new method \eqref{snbbmadp} clearly outperforms other compared methods.

%
\begin{table}[ht!b]
	{\footnotesize
		\caption{Distributions of $\{v_j:\,j=2,\ldots,n-1\}$.}\label{tbspe}
		\begin{center}
			\begin{tabular}{|c|c|}
				\hline
				\multirow{1}{*}{Set} &\multicolumn{1}{c|}{Spectrum} \\
				\hline
				\multirow{1}{*}{1} &$\{v_2,\ldots,v_{n-1}\}\subset(1,\kappa)$	\\
				\hline
				\multirow{2}{*}{2}
				&$\{v_2,\ldots,v_{n/5}\}\subset(1,100)$	\\
				&$\{v_{n/5+1},\ldots,v_{n-1}\}\subset(\frac{\kappa}{2},\kappa)$	\\
				\hline
				
				\multirow{2}{*}{3}
				&$\{v_2,\ldots,v_{n/2}\}\subset(1,100)$	\\
				&$\{v_{n/2+1},\ldots,v_{n-1}\}\subset(\frac{\kappa}{2},\kappa)$	\\
				\hline
				\multirow{2}{*}{4}
				&$\{v_2,\ldots,v_{4n/5}\}\subset(1,100)$	\\
				&$\{v_{4n/5+1},\ldots,v_{n-1}\}\subset(\frac{\kappa}{2},\kappa)$	\\
				\hline
				\multirow{3}{*}{5}
				&$\{v_2,\ldots,v_{n/5}\}\subset(1,100)$	\\
				&$\{v_{n/5+1},\ldots,v_{4n/5}\}\subset(100,\frac{\kappa}{2})$	\\
				&$\{v_{4n/5+1},\ldots,v_{n-1}\}\subset(\frac{\kappa}{2},\kappa)$	\\
				\hline
			\end{tabular}
		\end{center}
	}
\end{table}

\begin{figure}[ht!b]
	\centering
	\includegraphics[width=0.7\textwidth,height=0.5\textwidth]{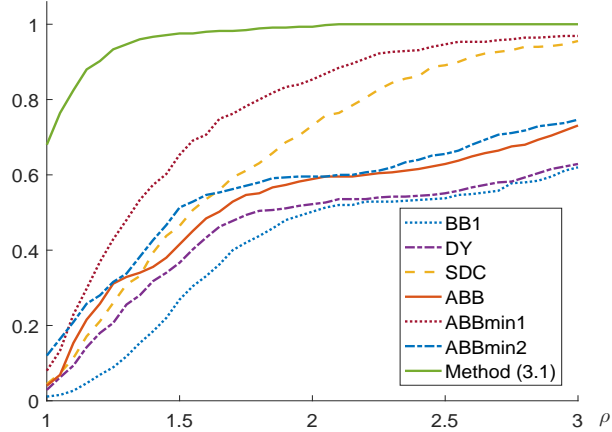}
	\caption{Performance profiles of method \eqref{snbbmadp}, BB1, DY, SDC, ABB, ABBmin1 and ABBmin2 on random quadratic problem \eqref{quad-test1} with spectral distributions in Table \ref{tbspe}, iteration metric.}\label{frandp}
\end{figure}

\begin{table}[htp!b]
	\setlength{\tabcolsep}{1ex}
	\caption{The average  number of iterations required by the method \eqref{snbbmadp}, the BB1, DY, SDC, ABB, ABBmin1 and ABBmin2 methods on quadratic problem \eqref{quad-test1}
		with spectral distributions in Table \ref{tbspe}.}\label{tbnbbrand}
	\centering
	\footnotesize
	\begin{tabular}{|c|c|c|c|c|c|c|c|c|}
		\hline
		\multicolumn{1}{|c|}{\multirow{1}{*}{set}} &\multicolumn{1}{c|}{\multirow{1}{*}{$\epsilon$}}
		&\multirow{1}{*}{BB1} &\multirow{1}{*}{DY} &\multirow{1}{*}{SDC} &\multirow{1}{*}{ABB}  &\multirow{1}{*}{ABBmin1} &\multirow{1}{*}{ABBmin2}
		&\multicolumn{1}{c|}{Method \eqref{snbbmadp}} \\
		\hline
		\multicolumn{1}{|c|}{\multirow{3}{*}{1}}
		& 1e-06   & 282.0  & 249.8  & \textbf{241.6}  & 256.1  & 268.1  & 306.5 & 247.9\\
		& 1e-09   &2849.6  &2598.7  &2069.6  &1322.4  &1864.2  &\textbf{1034.6} &1153.4\\
		& 1e-12   &5951.2  &6101.7  &3943.1  &1919.3  &4055.9  &\textbf{1436.3} &1970.0\\
		\hline
		
		\multicolumn{1}{|c|}{\multirow{3}{*}{2}}
		& 1e-06   & 348.3  & 277.3  & 182.2  & 273.7  & 146.6  & 273.5 & \textbf{105.7}\\
		& 1e-09   &1598.4  &1426.4  & 722.3  &1541.6  & 578.3  &1377.8 & \textbf{399.8}\\
		& 1e-12   &2848.4  &2676.6  &1309.2  &2782.8  & 974.9  &2223.5 & \textbf{666.6}\\
		\hline
		
		\multicolumn{1}{|c|}{\multirow{3}{*}{3}}
		& 1e-06   & 401.6  & 318.3  & 200.0  & 371.8  & 192.8  & 387.7 & \textbf{132.7}\\
		& 1e-09   &1850.3  &1495.6  & 783.5  &1623.1  & 615.9  &1445.5 & \textbf{418.5}\\
		& 1e-12   &3162.7  &2667.8  &1296.1  &2872.9  &1052.7  &2385.6 & \textbf{680.2}\\
		\hline
		
		\multicolumn{1}{|c|}{\multirow{3}{*}{4}}
		& 1e-06   & 498.2  & 456.1  & 264.0  & 443.4  & 214.7  & 503.5 & \textbf{152.3}\\
		& 1e-09   &1889.3  &1690.0  & 895.2  &1831.7  & 669.6  &1604.7 & \textbf{444.4}\\
		& 1e-12   &3037.2  &2827.8  &1382.9  &3005.0  &1105.7  &2458.5 & \textbf{704.8}\\
		\hline
		
		\multicolumn{1}{|c|}{\multirow{3}{*}{5}}
		& 1e-06   & 827.0  & 653.3  & 677.2  & 708.5  & 700.1  & 912.6 &\textbf{ 641.8}\\
		& 1e-09   &4006.8  &3770.8  &3358.5  &3151.8  &3079.7  &3232.2 &\textbf{2702.6}\\
		& 1e-12   &7549.6  &7649.3  &5749.2  &5186.1  &5350.0  &5180.8 &\textbf{4678.5}\\
		\hline
		
		\multicolumn{1}{|c|}{\multirow{3}{*}{total}}
		&1e-06   &2357.1 &1954.8 &1565.0 &2053.5 &1522.3 &2383.8 &\textbf{1280.4}       \\
		&1e-09    &12194.4 &10981.5 &7829.1 &9470.6 &6807.7 &8694.8 &\textbf{5118.7}    \\
		&1e-12    &22549.1 &21923.2 &13680.5 &15766.1 &12539.2 &13684.7 &\textbf{8700.1}\\
		
		\hline
		
	\end{tabular}
\end{table}

Table \ref{tbnbbrand} presents the average number of iterations required by the compared methods to meet given tolerances. We see that, for the first problem set, the new method \eqref{snbbmadp} is faster than the BB1, DY, SDC, ABB, and ABBmin1 methods, and is comparable to ABBmin2. For the second to fifth problem sets, the new method performs much better than all the others. In addition, for each value of $\epsilon$, the new method always outperforms other methods in the sense of total number of iterations.

Furthermore, we compared the above methods for the non-rand quadratic problem \cite{de2014efficient} with $A$ being a diagonal matrix given by
\begin{equation}\label{pro2}
A_{jj}=10^{\frac{\log_{10} \kappa}{n-1}(n-j)}, ~~~j=1,\ldots,n,
\end{equation}
and $b=0$. The problem dimensional was also set as $n=10000$. The setting of the parameters is the same as above except the pair $(h,s)$ used for the SDC method was set to $(30,\,2)$, which sounds to provide better performance than some other choices.

The average numbers of iterations over 10 different starting points with entries randomly generated in $[-10,10]$ are presented in Table \ref{tbqp}. From Table \ref{tbqp}, we can see that the method \eqref{snbbmadp} again outperforms the others for most of the problems.

\begin{table}[htp!b]
	\setlength{\tabcolsep}{1.ex}
	\caption{The average  number of iterations required by the method \eqref{snbbmadp}, the BB1, DY, SDC, ABB, ABBmin1 and ABBmin2 methods on problem \eqref{pro2} with $n=10000$.}\label{tbqp}
	\centering
	\begin{footnotesize}
		\begin{tabular}{|c|c|c|c|c|c|c|c|c|}
			\hline
			\multicolumn{1}{|c|}{\multirow{1}{*}{$\kappa$}} &\multicolumn{1}{c|}{\multirow{1}{*}{$\epsilon$}}
			&\multirow{1}{*}{BB1} &\multirow{1}{*}{DY}  &\multirow{1}{*}{SDC} &\multirow{1}{*}{ABB} &\multirow{1}{*}{ABBmin1} &\multirow{1}{*}{ABBmin2} &\multicolumn{1}{c|}{Method \eqref{snbbmadp}} \\
			\hline
			\multicolumn{1}{|c|}{\multirow{3}{*}{$10^4$}}
			& 1e-06   & 606.4      & 496.0      & 539.1      & 533.8      & 567.9      & 516.8   &\textbf{483.3}  \\
			& 1e-09   &1192.0      & 954.1      &1026.4      & 930.7      & 978.2      &\textbf{894.4}   & 951.9  \\
			& 1e-12   &1697.5      &1352.7      &1438.1      &1318.1      &1415.0      &\textbf{1310.6}   &1340.0  \\
			\hline
			
			\multicolumn{1}{|c|}{\multirow{3}{*}{$10^5$}}
			& 1e-06  &1476.9      &1254.4      &1204.6      &1288.8      &1198.5      &1243.5   &\textbf{1153.3}  \\
			& 1e-09  &3420.8      &3058.3      &2713.4      &2756.7      &2661.1      &2549.0   &\textbf{2580.3}  \\
			& 1e-12  &5532.7      &4871.3      &4163.6      &4024.4      &4027.6      &3948.5   &\textbf{3770.0}  \\
			\hline
			
			\multicolumn{1}{|c|}{\multirow{3}{*}{$10^6$}}
			& 1e-06  &2766.8      &2108.0      &1972.0      &2123.4      &2081.3      &3327.9   &\textbf{1903.0}  \\
			& 1e-09  &12792.1     &10719.3     &7049.2      &7889.9      &7956.8      &8057.4   &\textbf{6832.4}  \\
			& 1e-12  &18472.4     &18360.3     &11476.8     &12730.4     &12435.1     &12293.0  &\textbf{10999.2}  \\
			\hline
			
			\multicolumn{1}{|c|}{\multirow{3}{*}{total}}
			& 1e-06   &4850.1 &3858.4 &3715.7 &3946.0 &3847.7 &5088.2 &\textbf{3539.6}         \\
			& 1e-09   &17404.9 &14731.7 &10789.0 &11577.3 &11596.1 &11500.8 &\textbf{10364.6} \\
			& 1e-12   &25702.6 &24584.3 &17078.5 &18072.9 &17877.7 &17552.1 &\textbf{16109.2} \\
			\hline
		\end{tabular}
	\end{footnotesize}
\end{table}

\subsection{Unconstrained optimization}
To extend the new method \eqref{snbbmadp} for minimizing a general smooth function $f(x)$,
\begin{equation} \min_{x\in \mathbb{R}^n} \, f(x), \end{equation}
we usually need to incorporate some line search to ensure global convergence. As pointed out by Fletcher \cite{fletcher2005barzilai}, one important feature of BB-like methods is the inherent nonmonotone property. Some nonmonotone line search is often employed to gain good numerical performance \cite{dai2005projected,dai2001adaptive,raydan1997barzilai,zhang2004nonmonotone}. Here we would like to adopt the Grippo-Lampariello-Lucidi (GLL) nonmonotone line search \cite{grippo1986nonmonotone}, which accepts  $\lambda_{k}$ when it satisfies
\begin{equation} \label{nonmls}
f(x_{k}+\lambda_{k}d_{k})\leq f_{r}+\sigma\lambda_{k}g_{k} \tr d_{k},
\end{equation}
where $\sigma$ is a positive parameter and the reference value $f_{r}$ is the maximal function value in recentest available $M$ iterations; namely,
$f_{r}=\max_{0\leq i\leq \min\{k,M-1\}} f(x_{k-i})$. This strategy was firstly incorporated for unconstrained optimization in the  global BB (GBB) method by Raydan \cite{raydan1997barzilai} and performs well in practice.

The combination of the gradient method with the stepsize formula \eqref{snbbmadp} and the GLL nonmonotone line search yields a new algorithm, Algorithm \ref{alunc}, for unconstrained optimization.  Under standard assumptions, global convergence of Algorithm \ref{alunc} can similarly be established and the convergence rate is $R$-linear for strongly convex objective functions, see \cite{huang2015rate} for example.

\begin{algorithm}[ht!b]
	\caption{A gradient method for unconstrained optimization}\label{alunc}
	\begin{algorithmic}
		\STATE{Initialization: $x_{1}\in \mathbb{R}^n$, $\alpha_{\max}\geq\alpha_{\min}>0$, $\alpha_{1}\in[\alpha_{\min},\alpha_{\max}]$,  $\tau_{1}>0$, $\gamma>1$, $\epsilon, \sigma, \delta\in(0,1)$, $M\in\mathbb{N}$, $k:=1$. }
		
		\WHILE{$\|g_{k}\|_{\infty}>\epsilon$}
		\STATE{$d_{k}=-g_{k}$, $\lambda_{k}=\alpha_k$}
		\STATE{$f_{r}=\max_{0\leq i\leq \min\{k,M-1\}} f(x_{k-i})$}
		\WHILE{the condition \eqref{nonmls} does not hold}
		\STATE{$\lambda_{k}=\delta\lambda_{k}$}
		\ENDWHILE
		\STATE{$x_{k+1}=x_{k}+\lambda_{k}d_{k}$}
		\IF{$s_{k} \tr y_{k}>0$}
		\IF{ $\alpha_k^{BB2}/\alpha_k^{BB1}<\tau_k$ and $s_{k-1} \tr y_{k-1}>0$}
		\IF{ $\alpha_{k+1}^{new}>0$}
		\STATE{${\alpha}_{k+1}=\min\{\alpha_{k}^{BB2}, \alpha_{k+1}^{BB2}, \alpha_{k+1}^{new}\}$}
		\ELSE	
		\STATE{$\alpha_{k+1} =\min\{\alpha_{k}^{BB2}, \alpha_{k+1}^{BB2}\}$}
		\ENDIF
		\STATE{$\tau_{k+1}=\tau_k/\gamma$}
		\ELSE	
		\STATE{$\alpha_{k+1} = \alpha_{k+1}^{BB1}$}	
		\STATE{$\tau_{k+1}=\tau_k\gamma$}
		\ENDIF
		\ELSE	
		\STATE{$\alpha_{k+1} = \min\{1/\|g_{k}\|_\infty,\|x_{k}\|_\infty/\|g_{k}\|_\infty\}$}
		\ENDIF
		\STATE{Chop extreme values of the stepsize such that $\alpha_{k+1}\in[\alpha_{\min},\alpha_{\max}]$}
		\STATE{$k=k+1$}
		\ENDWHILE
	\end{algorithmic}
\end{algorithm}

We compared Algorithm \ref{alunc} with the GBB \cite{raydan1997barzilai} and ABBmin \cite{Serafino2018,frassoldati2008new} methods for unconstrained optimization problems from the CUTEst collection \cite{gould2015cutest} with dimension less than or equal to $10000$. We deleted the problem if either it can not be solved in $200000$ iterations by any of the algorithms or the function evaluation exceeds one million and $147$ problems are left.

\begin{figure}[htp!b]
	\centering
	\subfloat{\label{fig:2a}\includegraphics[width=0.47\textwidth,height=0.35\textwidth]{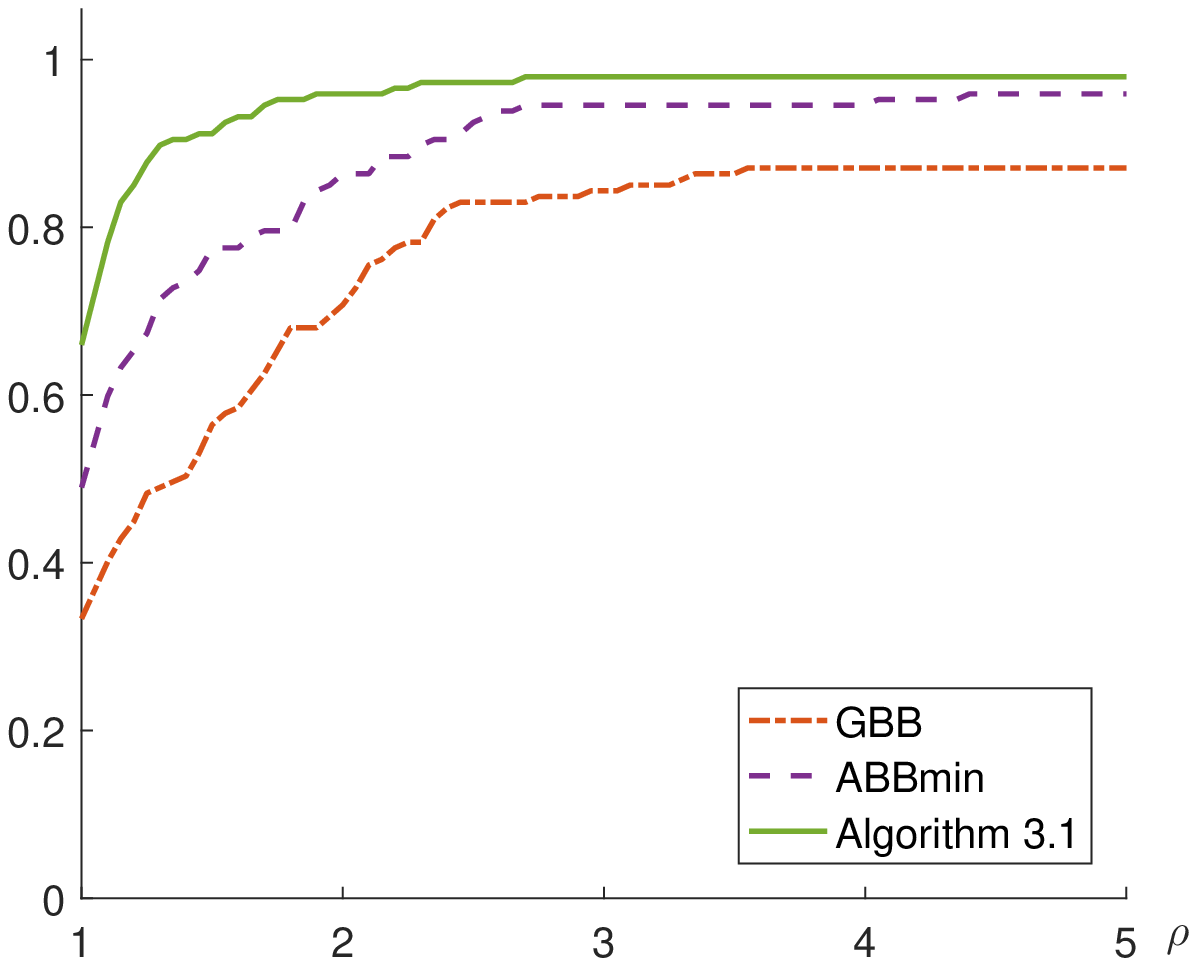}}
	\subfloat{\label{fig:2b}\includegraphics[width=0.47\textwidth,height=0.35\textwidth]{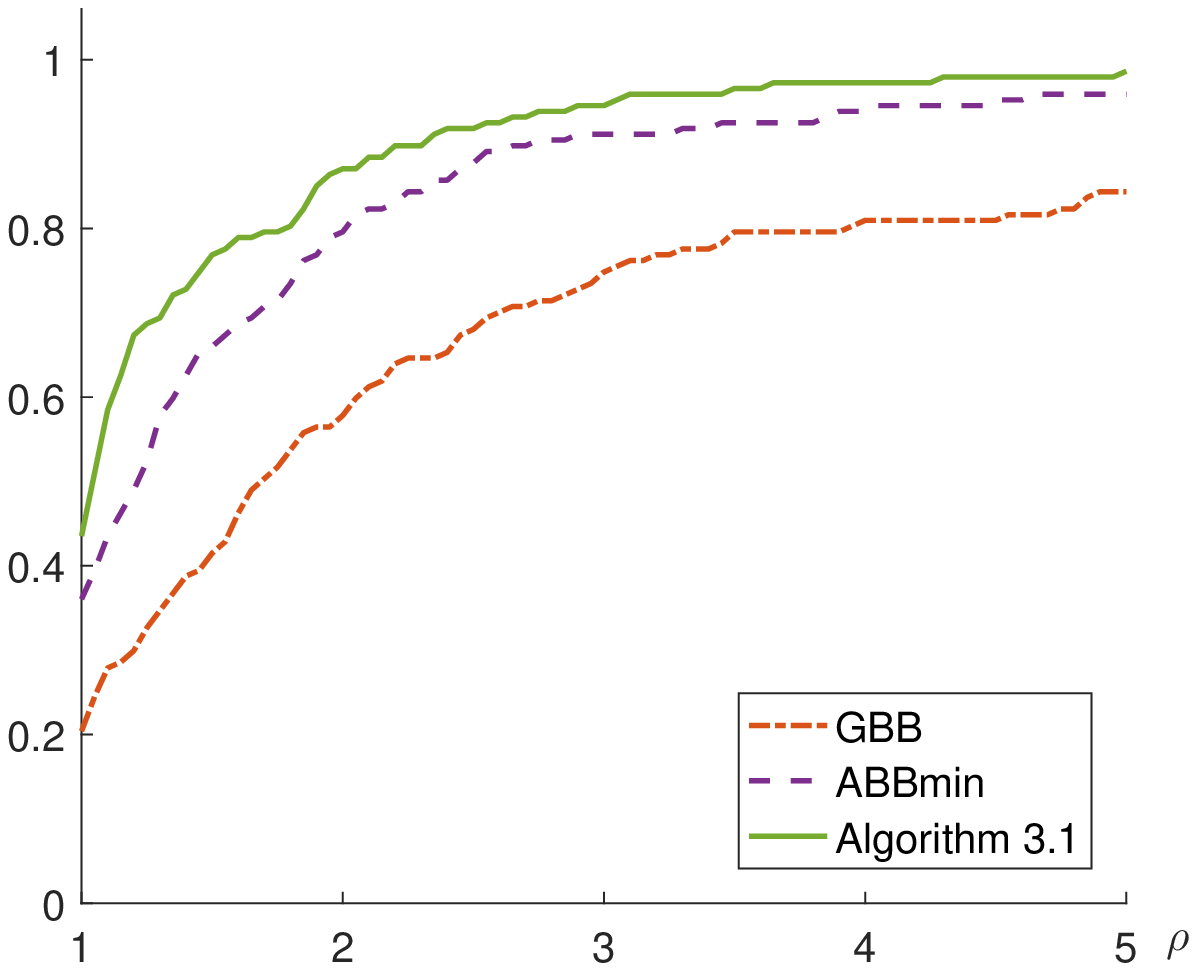}}
	\caption{Performance profiles of Algorithm \ref{alunc}, GBB and ABBmin on 147 unconstrained problems from CUTEst, iteration (left) and CPU time (right) metrics.}\label{cutestunc}
\end{figure}

To implement Algorithm \ref{alunc}, we use $\alpha_{1}=\|x_{1}\|_\infty/\|g_{1}\|_\infty$ if $x_{1}\neq0$ and otherwise $\alpha_{1}=1/\|g_{1}\|_\infty$. In addition, we set $\alpha_{2}=\alpha_{2}^{BB1}$ for the case $\alpha_{2}^{BB1}>0$ and otherwise choose $\alpha_{2} = \min\{1/\|g_{2}\|_\infty,\|x_{2}\|_\infty/\|g_{2}\|_\infty\}$.
The following parameters were employed  for Algorithm \ref{alunc}:
$$\alpha_{\min}=10^{-10},\, \alpha_{\max}=10^{6},\, M=10,\, \sigma=10^{-4},\,\delta=0.5,\, \tau_1=0.2, \, \gamma=1.02.$$
Default parameter settings were used for the GBB and ABBmin methods. The stopping condition $\|g_{k}\|_\infty\leq10^{-6}$ was adopted for all the three methods.

Performance profiles of Algorithm \ref{alunc}, GBB and ABBmin using iteration and CPU time metrics are plotted in Figure \ref{cutestunc}. From Figure \ref{cutestunc}, it can be seen that Algorithm \ref{alunc} performs significantly better than GBB and better than ABBmin.

\section{Extreme eigenvalues computation}\label{secegen}
In this section, we consider the problem of computing several extreme eigenvalues of large-scale real symmetric  matrices.

For a given $n\times n$ real symmetric positive definite matrix $A$, we are interested in the first $r\ll n$ largest/smallest eigenvalues and their corresponding eigenvectors, which has important applications in scientific and engineering computing such as principal component analysis \cite{daspremont2007a} and electronic structure calculation \cite{hu2018structured}. This problem can be formulated as an unconstrained optimization problem \cite{absil2002a}
\begin{equation}\label{egenray}
\min_{X\in\mathbb{R}^{n\times r}} \quad \textrm{tr}(X \tr AX(X \tr X)^{-1})
\end{equation}
or a constrained optimization problem with orthogonality constraints \cite{sameh2000the,sameh1982a}
\begin{equation}\label{egentrc}
\min_{X\in\mathbb{R}^{n\times r}} \quad \textrm{tr}(X \tr AX) \quad\textrm{s.t.} \quad X \tr X=I_r,
\end{equation}
where $I_r$ denotes the $r\times r$ identity matrix. However, it is not easy to calculate the inverse or orthogonalization of a matrix, especially in the case of large dimension. To avoid these difficulties, Jiang et al. \cite{jiang2014} provides the following equivalent unconstrained reformulation
\begin{equation}\label{egenvp}
\min_{X\in\mathbb{R}^{n\times r}} P_\mu(X)=\frac{1}{4}tr(X \tr XX \tr X)+ \frac{1}{2}tr(X \tr (A-\mu I_n)X),
\end{equation}
where $\mu>0$ is a scaling parameter. They proposed the so-called EigUncABB method for solving \eqref{egenvp}, which is very competitive with the Matlab function EIGS and other recent methods.

In order to apply Algorithm \ref{alunc} to problem \eqref{egenvp}, we replace the two BB stepsizes $\alpha_k^{BB1}$
and $\alpha_k^{BB2}$ in the calculations of $\alpha_{k}^{new}$ and \eqref{snbbmadp} by
\begin{equation*}
\alpha_{k}^{MBB1}=\frac{\textrm{tr}(S_{k-1} \tr S_{k-1})}{\textrm{tr}(S_{k-1} \tr Y_{k-1})}~~\textrm{and}~~
\alpha_{k}^{MBB2}=\frac{\textrm{tr}(S_{k-1} \tr Y_{k-1})}{\textrm{tr}(Y_{k-1} \tr Y_{k-1})},
\end{equation*}
respectively, where $S_{k-1}=X_k-X_{k-1}$ and $Y_{k-1}=\nabla P_\mu(X_k)-\nabla P_\mu(X_{k-1})$. Notice that the above two modified BB stepsizes are different from those employed by EigUncABB, which uses $|\alpha_{k}^{MBB1}|$ and $|\alpha_{k}^{MBB2}|$ in an alternate manner. We shall refer to the modified $\alpha_{k}^{new}$ as $\alpha_{k}^{Mnew}$.

Since the evaluation of $P_\mu(X)$ is expensive for large-dimension $X$, we shall adopt the Dai-Fletcher nonmonotone line search \cite{dai2005projected} to reduce the number of function evaluations. In particular, it uses $f_{best}$, $f_c$, $m$ and $M$ to update the reference value $f_{r}$ in \eqref{nonmls}, where $f_{best}=\min_{1\leq j\leq k}~f(x_j)$ is the current best function value, $f_c$ is the maximum value of the objective function since the value of $f_{best}$ was found, $m$ is the number of iterations since the value of $f_{best}$ was obtained, and $M$ is a preassigned number. The value of $f_{r}$ is unchanged if the method find a better function value in $M$ iterations. Otherwise, if $m=M$, $f_{r}$ is reset to $f_c$ and $f_c$ is reset to the current function value. See \cite{dai2005projected} for details on this line search.

Our method for problem \eqref{egenvp} is formally presented in Algorithm \ref{aleig}. The global convergence can be established similarly to that for EigUncABB.

\begin{algorithm}[ht!b]
	\caption{A gradient method for extreme eigenvalues problems}\label{aleig}
	\begin{algorithmic}
		\STATE{Initialization: $X_{1}\in \mathbb{R}^{n\times r}$,  $\tau_{1}>0$,  $\alpha_{\max}\geq\alpha_{\min}>0$, $\alpha_{1}\in[\alpha_{\min},\alpha_{\max}]$,  $\tau_{1}>0$, $\gamma>1$, $\epsilon, \sigma, \delta\in(0,1)$, $M\in\mathbb{N}$, $m=0$, $f_{r}=f_{best}=f_c=P_\mu(X_1)$, $k:=1$.  }
		
		\WHILE{$\|\nabla P_\mu(X_k)\|_{F}>\epsilon$}
		\IF{$P_\mu(X_k)< f_{best}$}
		\STATE{$f_{best}=P_\mu(X_k)$, $f_{c}=P_\mu(X_k)$, $m=0$}
		\ELSE	
		\STATE{$f_{c}=\max\{f_{c},P_\mu(X_k)\}$, $m=m+1$}
		\IF{$m=M$}
		\STATE{$f_{r}=P_\mu(X_k)$, $f_{c}=P_\mu(X_k)$, $m=0$}
		\ENDIF
		\ENDIF
		\STATE{$d_{k}=-\nabla P_\mu(X_k)$, $\lambda_{k}=\alpha_k$}
		\WHILE{the condition \eqref{nonmls} does not hold}
		\STATE{$\lambda_{k}=\delta\lambda_{k}$}
		\ENDWHILE
		\STATE{$X_{k+1}=X_{k}+\lambda_{k}d_{k}$}
		\IF{$\textrm{tr}(S_{k} \tr Y_{k})>0$}	
		\IF{ $\alpha_k^{MBB2}/\alpha_k^{MBB1}<\tau_k$ and $\textrm{tr}(S_{k-1} \tr Y_{k-1})>0$}
		\IF{ $\alpha_{k+1}^{Mnew}>0$}
		\STATE{${\alpha}_{k+1}=\min\{\alpha_{k}^{MBB2}, \alpha_{k+1}^{MBB2}, \alpha_{k+1}^{Mnew}\}$}
		\ELSE	
		\STATE{$\alpha_{k+1} =\min\{\alpha_{k}^{MBB2}, \alpha_{k+1}^{MBB2}\}$}		
		\ENDIF
		\STATE{$\tau_{k+1}=\tau_k/\gamma$}
		\ELSE	
		\STATE{$\alpha_{k+1} = \alpha_{k+1}^{MBB1}$}	
		\STATE{$\tau_{k+1}=\tau_k\gamma$}
		\ENDIF
		\ELSE	
		\STATE{$\alpha_{k+1} = |\alpha_{k+1}^{MBB1}|$}
		\ENDIF
		\STATE{Chop extreme values of the stepsize such that $\alpha_{k+1}\in[\alpha_{\min},\alpha_{\max}]$}
		\STATE{Compute $\tau_{k+1}$ by \eqref{tauk}}
		\STATE{$k=k+1$}
		\ENDWHILE
	\end{algorithmic}
\end{algorithm}

We compared Algorithm \ref{aleig} with EigUncABB on extreme eigenvalues problems which involve a $16,000\times16,000$ matrix, say $A$, generated by the \emph{laplacian} function in Matlab. The matrix $A$ can be viewed as the 3D negative Laplacian on a rectangular finite difference grid.

For Algorithm \ref{aleig}, we choose $\alpha_1=\|\nabla P_\mu(X_1)\|_F^{-1}$ and $\alpha_{2} = |\alpha_{2}^{MBB1}|$.
The other parameters in Algorithm \ref{aleig} were chosen in the same way as in the above section and default parameters for EigUncABB were employed. As suggested in \cite{jiang2014}, $\mu$ was initialized to $1.01\times\lambda_{\bar{r}}(X_1 \tr AX_1)$, where $\bar{r}=\max\{\lfloor1.1r\rfloor,10\}$ with $\lfloor\cdot\rfloor$ denoting the nearest integer less than or equal to the corresponding element. Moreover, it is updated by $\mu=1.01\lambda_{\bar{r}}(X_k\tr AX_k)$ when $\|\nabla P_\mu(X_k)\|_F\leq 0.1^{j}\|\nabla P_\mu(X_1)\|_F$ and $j\leq j_{\max}$ for some integers $j$ and $j_{\max}$. Specifically, the value of $j_{\max}$ was set to $3$.
The initial value of $j$ was set to $1$ and it will be increased by one if $\mu$ is updated. We generated initial points by the following Matlab codes
\begin{equation*}
\textrm{seed} = 100;\, \textrm{rng(seed,`twister')};\, X_1 = \textrm{randn}(n,r); \,  X_1 = \textrm{orth}(X_1).
\end{equation*}
Both methods were terminated provided $\|\nabla P_\mu(X_k)\|_F\leq10^{-3}$.

To measure the quality of computed solutions, we calculate the relative eigenvalue error and residual error of the $i$-th eigenpair as
\begin{equation*}
\textrm{err}_i=\frac{|\bar{\lambda}_i-\lambda_i|}{\max\{1,|\lambda_i|\}}~~\textrm{and}~~
\textrm{resi}_i=\frac{|A\bar{u}_i-\bar{\lambda}_i\bar{u}_i|}{\max\{1,|\bar{\lambda}_i|\}},
\end{equation*}
respectively. Here ${\lambda}_i$ is the true $i$-th eigenvalue, $\bar{u}_i$ and $\bar{\lambda}_i$ are the $i$-th eigenvector and eigenvalue obtained by the compared algorithms, respectively.

\begin{table}[th!b]
	\setlength{\tabcolsep}{1.4ex}
	\caption{Results of EigUncABB and Algorithm \ref{aleig} on extreme eigenvalue problems.}\label{tbeigcmp}
	\footnotesize
	\begin{center}
		\begin{tabular}{|c|c|c|c|c|c|c|c|c|c|c|}
			\hline
			\multicolumn{1}{|c|}{\multirow{1}{*}{}}
			&\multicolumn{5}{c|}{EigUncABB}  &\multicolumn{5}{c|}{Algorithm \ref{aleig}}\\
			\cline{2-11}
			$r$   &resi  &err  &iter &nfe  &time    &resi  &err  &iter  &nfe  &time \\
			\hline		
			20     &1.4e-05  &9.8e-09  &140  &160  &2.25  	&4.2e-06   &3.3e-09  &136  &144  &2.10   \\
			50     &1.6e-05  &1.8e-08  &146  &170  &4.58    &2.3e-05   &2.1e-07  &128  &133  &3.90   \\
			100    &7.1e-05  &2.0e-09  &167  &183  &10.39   &1.4e-04   &1.4e-08  &137  &146  &8.70   \\
			200    &3.8e-06  &5.2e-10  &192  &216  &28.97   &1.1e-06   &5.4e-10  &168  &175  &24.26  \\
			300    &4.1e-06  &8.1e-11  &160  &188  &44.98   &4.9e-06   &2.4e-11  &149  &163  &39.91  \\
			400    &6.9e-07  &6.7e-13  &148  &168  &64.40   &4.0e-06   &3.7e-11  &166  &178  &67.94  \\
			500    &3.2e-06  &2.9e-11  &201  &227  &122.32  &6.8e-07   &7.7e-14  &171  &188  &101.95 \\
			600    &3.2e-06  &1.6e-12  &204  &238  &169.76  &4.4e-06   &6.4e-12  &172  &190  &137.86 \\
			700    &6.0e-07  &1.1e-12  &185  &212  &199.48  &2.6e-06   &3.3e-12  &166  &185  &177.61 \\
			800    &2.4e-06  &1.9e-11  &208  &234  &270.53  &1.3e-05   &4.7e-11  &194  &211  &244.68 \\
			900    &3.5e-06  &7.3e-12  &171  &194  &268.97  &1.0e-07   &1.9e-13  &164  &175  &243.93 \\
			1000   &1.3e-06  &2.1e-12  &211  &240  &388.06  &1.1e-06   &9.7e-13  &194  &208  &341.42 \\
			\hline
		\end{tabular}
	\end{center}
\end{table}

In Table \ref{tbeigcmp}, ``resi'' denotes mean values of resi$_i$, ``err'' denotes mean values of err$_i$, $i=1,\cdots,r$, ``iter'' is the number of iterations, ``nfe'' is the total number of function evaluations and ``time'' is the CPU time in seconds. From Table \ref{tbeigcmp}, we can see that Algorithm \ref{aleig} is comparable to  EigUncABB in the sense of relative eigenvalue error and residual error. Moreover, Algorithm \ref{aleig} outperforms EigUncABB in terms of iterations, function evaluations and CPU time for most values of $r$.

\section{Special constrained optimization}\label{alcon}
In this section, we extend Algorithm \ref{alunc} to solve two special constrained optimization of the form
\begin{equation}\label{conprob}
\min_{x\in\Omega} ~~f(x),
\end{equation}
where $\Omega \subseteq \mathbb{R}^n$ is a closed convex set and $f$ is a Lipschitz continuously differentiable function on $\Omega$.

Our algorithm for solving problem \eqref{conprob} falls into the gradient projection category, which calculates the search direction by
\begin{equation}\label{dirc}
d_k=P_{\Omega}(x_k-\alpha_kg_k)-x_k,
\end{equation}
with $P_{\Omega}(\cdot)$ being the Euclidean projection onto $\Omega$ and $\alpha_k$ being the stepsize. For a general closed convex set $\Omega$, the projection $P_{\Omega}(\cdot)$ may not be easy to compute. However, for the box-constrained optimization, where $\Omega=\{x\in\mathbb{R}^n|~l\leq x\leq u\}$, we have $P_{\Omega}(x)=\max\{l,\min\{x,u\}\}$. Here, $l \le x \le u$ means componentwise; namely, $l_i \le x_i \le u_i$ for all $i=1, \ldots, n$. In addition, for singly linearly box-constrained optimization, the projection to its feasible
set $\Omega=\{x\in\mathbb{R}^n|~l\leq x\leq u,~a \tr x=b\}$ with $a\in \mathbb{R}^n$ and $b\in \mathbb{R}$ can efficiently be computed by for example the secant-like algorithm developed in \cite{dai2006new}. In what follows, we will focus on box-constrained optimization and singly linearly box-constrained optimization.

For the calculation of the stepsize $\alpha_k$, we can simply utilize the formula \eqref{snbbmadp}. However, this unmodified stepsize cannot capture the feature of the constraints well in the context of constrained optimization. Therefore we will employ the following stepsize
\begin{equation}\label{newgmd1}
{\alpha}_{k}=
\begin{cases}
\min\{\bar{\alpha}_{k-1}^{BB2},\bar{\alpha}_{k}^{BB2},\bar{\alpha}_{k}^{new}\},& \text{if $\bar{\alpha}_{k}^{BB2}/\bar{\alpha}_{k}^{BB1}<\tau_k$;} \\
\bar{\alpha}_{k}^{BB1},& \text{otherwise},
\end{cases}
\end{equation}
where $\tau_k$ is updated by the rule \eqref{tauk}, $\bar{\alpha}_{k}^{BB1}$ and $\bar{\alpha}_{k}^{BB2}$ are two modified BB stepsizes specified in the following subsections, and the stepsize $\bar{\alpha}_{k}^{new}$ is defined by replacing $\alpha_{k}^{BB1}$ and $\alpha_{k}^{BB2}$ in $\alpha_{k}^{new}$ with $\bar{\alpha}_{k}^{BB1}$ and $\bar{\alpha}_{k}^{BB2}$, respectively.

We present our method for problem \eqref{conprob} in Algorithm \ref{al1}. For the first two stepsizes in Algorithm \ref{al1}, we employ the same choices as  Algorithm \ref{alunc} but replace $\|g_{j}\|_\infty$ by $\|P_{\Omega}(x_{j}-g_{j})-x_{j}\|_{\infty}$ for $j=1,2$ and $\alpha_{2}^{BB1}$ by $\bar{\alpha}_{2}^{BB1}$, respectively. Global convergence and $R$-linear convergence of Algorithm \ref{al1} can be established as in \cite{huang2015rate}.

\begin{algorithm}[ht!b]
	\caption{Projected gradient method for problem \eqref{conprob}}\label{al1}
	\begin{algorithmic}
		\STATE{Initialization: $x_{1}\in \mathbb{R}^n$, $\alpha_{\max}\geq\alpha_{\min}>0$, $\alpha_{1}\in[\alpha_{\min},\alpha_{\max}]$,  $\tau_{1}>0$, $\gamma>1$, $\epsilon, \sigma, \delta\in(0,1)$, $M\in\mathbb{N}$, $k:=1$. }
		
		\WHILE{$\|P_{\Omega}(x_{k}-g_{k})-x_{k}\|_{\infty}>\epsilon$}
		\STATE{Compute the search direction $d_k$ by \eqref{dirc}, $\lambda_{k}=\alpha_k$}
		\STATE{$f_{r}=\max_{0\leq i\leq \min\{k,M-1\}} f(x_{k-i})$}
		\WHILE{the condition \eqref{nonmls} does not hold}
		\STATE{$\lambda_{k}=\delta\lambda_{k}$}
		\ENDWHILE
		\STATE{$x_{k+1}=x_{k}+\lambda_{k}d_{k}$}
		\IF{$s_{k} \tr \bar{y}_{k}>0$}
		\IF{ $\bar{\alpha}_k^{BB2}/\bar{\alpha}_k^{BB1}<\tau_k$ and $s_{k-1} \tr \bar{y}_{k-1}>0$}
		\IF{ $\bar{\alpha}_{k+1}^{new}>0$}
		\STATE{${\alpha}_{k+1}=\min\{\bar{\alpha}_{k}^{BB2}, \bar{\alpha}_{k+1}^{BB2}, \bar{\alpha}_{k+1}^{new}\}$}
		\ELSE	
		\STATE{$\alpha_{k+1} =\min\{\bar{\alpha}_{k}^{BB2}, \bar{\alpha}_{k+1}^{BB2}\}$}
		\ENDIF
		\STATE{$\tau_{k+1}=\tau_k/\gamma$}
		\ELSE	
		\STATE{$\alpha_{k+1} = \bar{\alpha}_{k+1}^{BB1}$}	
		\STATE{$\tau_{k+1}=\tau_k\gamma$}
		\ENDIF
		\ELSE	
		\STATE{$\alpha_{k+1} = \min\{1/\|P_{\Omega}(x_{k}-g_{k})-x_{k}\|_\infty,\|x_{k}\|_{\infty}/\|P_{\Omega}(x_{k}-g_{k})-x_{k}\|_{\infty}\}$}
		\ENDIF
		\STATE{Chop extreme values of the stepsize such that $\alpha_{k+1}\in[\alpha_{\min},\alpha_{\max}]$}
		\STATE{$k=k+1$}
		\ENDWHILE
	\end{algorithmic}
\end{algorithm}

\subsection{Box-constrained optimization}


As pointed out by \cite{huang2019gradient,hdlz2019}, for the box-constrained problem, a projected gradient method often eventually solves an unconstrained problem in the subspace corresponding to free variables. So, it is useful to modify the stepsize by considering those free variables only. To this end, we consider to replace the gradients of $f$ in  the two BB stepsizes by those of the Lagrangian function
\begin{equation*}\label{lagrange}
\mathcal{L}(x,\delta,\zeta)=f(x)-\delta \tr (x-l)-\zeta \tr (u-x),
\end{equation*}
where $\delta,\zeta\in\mathbb{R}^n$ are Lagrange multipliers. That is,
\begin{equation}\label{bsbb-1-2}
\bar{\alpha}_k^{BB1}  =\frac{s_{k-1} \tr s_{k-1}}{s_{k-1} \tr \bar{y}_{k-1}} \quad \mbox{and} \quad  \bar{\alpha}_k^{BB2}=\frac{s_{k-1} \tr \bar{y}_{k-1}}{\bar{y}_{k-1} \tr \bar{y}_{k-1}},
\end{equation}
where
\begin{align}\label{barybox}
\bar{y}_{k-1}&=\nabla_x\mathcal{L}(x_k,\delta_k,\zeta_k)-\nabla_x\mathcal{L}(x_{k-1},\delta_{k-1},\zeta_{k-1})\nonumber\\
&=y_{k-1}-(\delta_k-\delta_{k-1})+(\zeta_k-\zeta_{k-1}).
\end{align}
In this way, the above two modified BB stepsizes take the constraints into consideration. Use the sets $\mathcal{A}_k$ and $\mathcal{I}_k=\mathcal{N}\setminus\mathcal{A}_k$ to estimate the active and inactive sets at $x_k$, respectively, where $\mathcal{N}=\{1,2,\ldots,n\}$. Based on the above analysis, we simply set $\bar{y}_{k-1}^{(i)}=0$ for $i\in\mathcal{A}_k$. Notice that, by the first-order optimality conditions of problem \eqref{conprob}, the Lagrange multipliers for free variables are zeros. Then we set $\delta_k^{(i)}-\delta_{k-1}^{(i)}=0$ and $\zeta_k^{(i)}-\zeta_{k-1}^{(i)}=0$ for $i\in\mathcal{I}_k$.
Hence, $\bar{y}_{k-1}$ can be written as
\begin{equation} \label{bary}
\bar{y}_{k-1}^{(i)}=\left\{
\begin{array}{ll}
0, & \hbox{ if $i\in\mathcal{A}_k$;} \\
g_k^{(i)}-g_{k-1}^{(i)}, & \hbox{ otherwise.}
\end{array}
\right.
\end{equation}
In our test, we set $\mathcal{A}_k=\{i\in\mathcal{N}~|~s_{k-1}^{(i)}=0\}$, which is suitable for box-constrained optimization. The above two modified BB stepsizes often yield better performance than the original BB stepsizes, see for example \cite{huang2019gradient,hdlz2019}. Here we mention that similar modified BB stepsizes are presented in \cite{Crisci2019} for box-constrained quadratic programming.

Now we compare Algorithm \ref{al1} with SPG\footnote{codes available at \url{https://www.ime.usp.br/~egbirgin/tango/codes.php}} \cite{birgin2000nonmonotone,birgin2014spectral} and BoxVABBmin \cite{Crisci2019} for box-constrained problems from the CUTEst collection \cite{gould2015cutest} with dimension larger than $50$. Notice that SPG is a generalization of GBB with the long BB stepsize $\alpha_k^{BB1}$ and BoxVABBmin is a variant of ABBmin with the modified BB stepsizes \eqref{bsbb-1-2} using a different  $\mathcal{A}_k$. Three of the problems were deleted since none of the three algorithms can solve them successfully and hence $47$ problems are left in our test.

\begin{figure}[thp!b]
	\centering
	\subfloat{\label{fig:3a}\includegraphics[width=0.47\textwidth,height=0.35\textwidth]{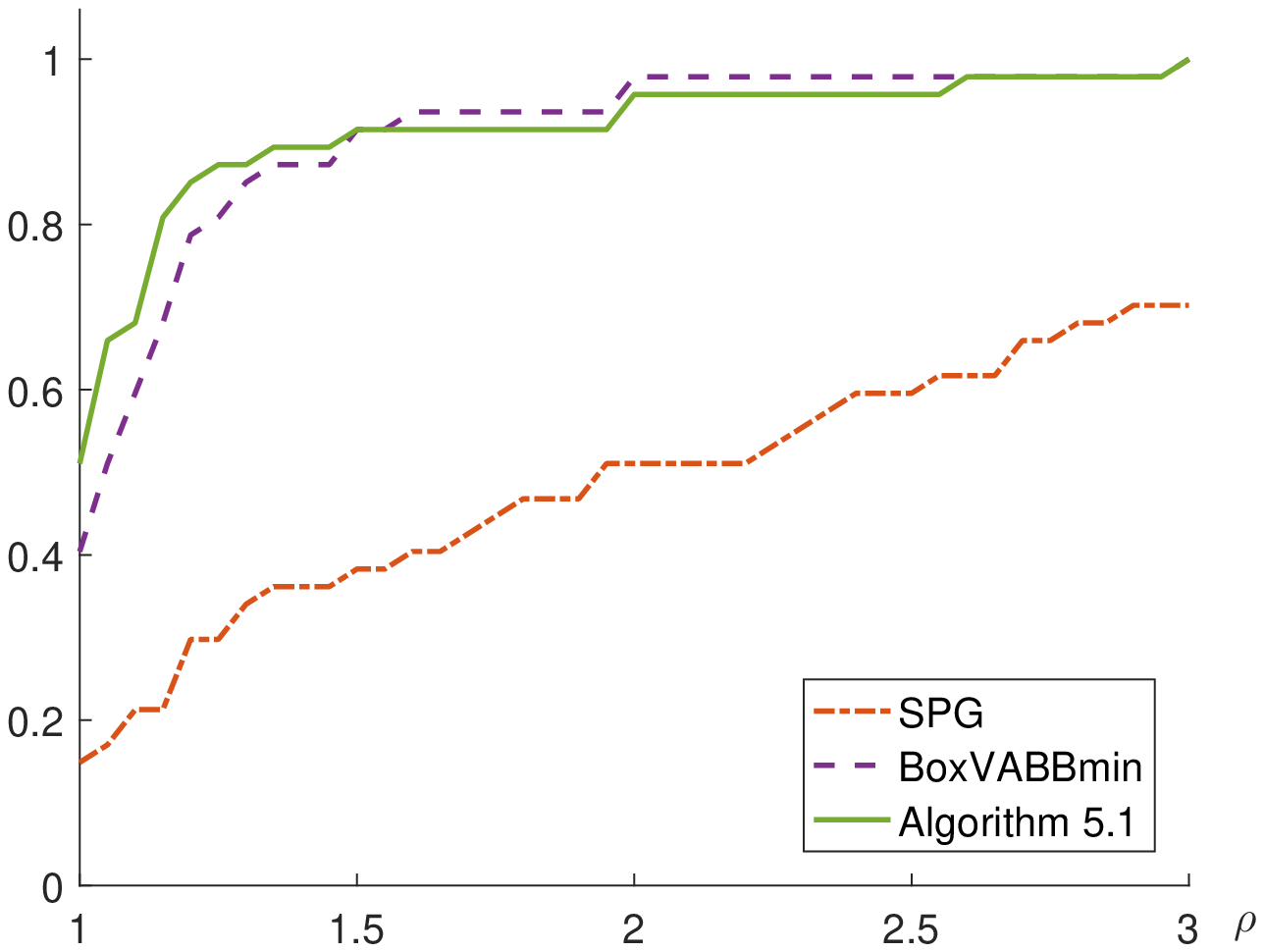}}
	\subfloat{\label{fig:3b}\includegraphics[width=0.47\textwidth,height=0.35\textwidth]{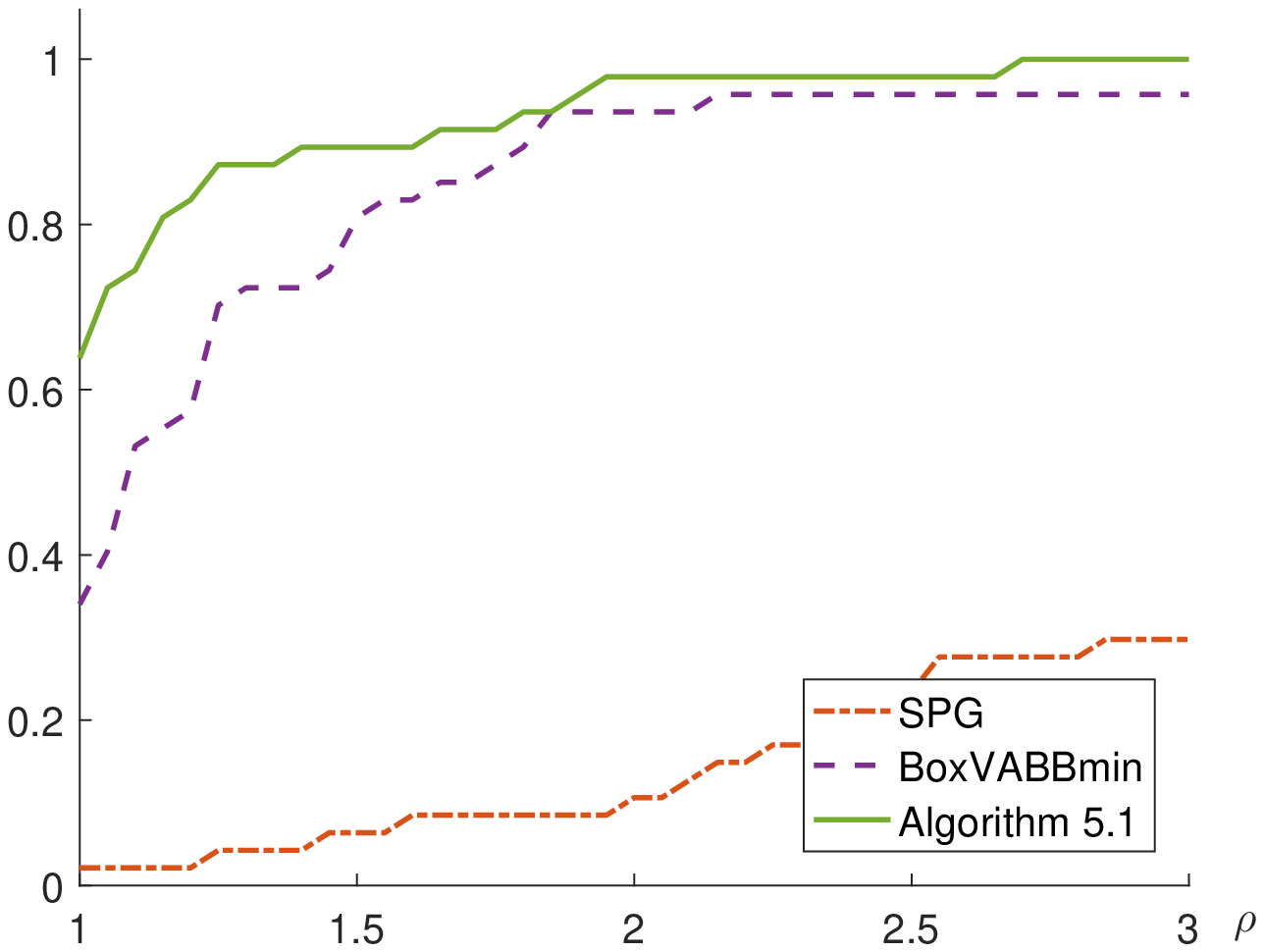}}
	\caption{Performance profiles of Algorithm \ref{al1}, SPG and BoxVABBmin for the remaining $47$ box-constrained problems from CUTEst, iteration (left) and CPU time (right) metrics.}\label{cutest}
\end{figure}

We adopted default parameters for SPG and BoxVABBmin and used the same settings for Algorithm \ref{al1} as for the unconstrained case. The algorithms were terminated if either $\|P_{\Omega}(x_{k}-g_{k})-x_{k}\|_\infty\leq10^{-6}$ or the number of iterations exceeds $200000$.

Figure \ref{cutest} presents the performance profiles of Algorithm \ref{al1}, SPG and BoxVABBmin based on iteration  and CPU time metrics. Clearly, Algorithm \ref{al1} and BoxVABBmin perform much better than SPG while Algorithm \ref{al1} outperforms BoxVABBmin in both number of iterations and CPU time.

\subsection{Singly linearly box-constrained optimization}
Now we shall consider the solution of the singly linearly box-constrained (SLB) optimization problem.
To improve the BB stepsizes \eqref{bsbb-1-2} by taking the constraints into consideration, similarly to \eqref{barybox}, denote
\begin{align}
\bar{y}_{k-1}&=\nabla_x\mathcal{L}(x_k,\delta_k,\zeta_k,t_k)-\nabla_x\mathcal{L}(x_{k-1},\delta_{k-1},\zeta_{k-1},t_{k-1})\nonumber\\
&=y_{k-1}-(\delta_k-\delta_{k-1})+(\zeta_k-\zeta_{k-1})-(t_k-t_{k-1})a,
\end{align}
where
\begin{equation*}\label{lagrange-slb}
\mathcal{L}(x,\delta,\zeta,t)=f(x)-\delta \tr (x-l)-\zeta \tr (u-x)-t(a \tr x-b).
\end{equation*}
Similarly to the box-constrained case, we set $\bar{y}_{k-1}^{(i)}=0$ for $i\in\mathcal{A}_k$, and $\delta_k^{(i)}-\delta_{k-1}^{(i)}=0$ and $\zeta_k^{(i)}-\zeta_{k-1}^{(i)}=0$ for $i\in\mathcal{I}_k$. Again by the first-order optimality conditions of problem \eqref{conprob}, $\nabla_x\mathcal{L}$ is zero at the solution, which yields
$$\bar{y}_{k-1}^{(\mathcal{I}_k)}=y_{k-1}^{(\mathcal{I}_k)}-(t_k-t_{k-1})a^{(\mathcal{I}_k)}=0.$$
It is necessary to estimate the Lagrange multipliers $t_k$ and $t_{k-1}$ for computing $\bar{y}_{k-1}^{(\mathcal{I}_k)}$. A simple way is to estimate $t_k-t_{k-1}$ directly by
$$t_k-t_{k-1}=\frac{(a^{(\mathcal{I}_k)}) \tr y_{k-1}^{(\mathcal{I}_k)}}{(a^{(\mathcal{I}_k)}) \tr a^{(\mathcal{I}_k)}}.$$
Thus, we have
\begin{equation} \label{baryslb}
\bar{y}_{k-1}^{(i)}=\left\{
\begin{array}{ll}
0, & \hbox{ if $i\in\mathcal{A}_k$;} \\
y_{k-1}^{(\mathcal{I}_k)}-\frac{(a^{(\mathcal{I}_k)}) \tr y_{k-1}^{(\mathcal{I}_k)}}{(a^{(\mathcal{I}_k)}) \tr a^{(\mathcal{I}_k)}}a^{(\mathcal{I}_k)}, & \hbox{ otherwise.}
\end{array}
\right.
\end{equation}
This together with  \eqref{bsbb-1-2} provides us the modified BB stepsizes for SLB problems. For the case that
\begin{equation}\label{actsetlsb}
\mathcal{A}_k=\{i\in\mathcal{N}~|~x_k^{(i)}=x_{k-1}^{(i)}=l_i~~\mathrm{or}~~x_k^{(i)}=x_{k-1}^{(i)}=u_i\}
\end{equation}
(see \cite{Serena2020}), the stepsizes in \eqref{bsbb-1-2} are derived in a different way.

In what follows, we compare Algorithm \ref{al1} with the Dai-Fletcher method \cite{dai2006new} and the EQ-VABBmin method \cite{Serena2020} for random SLB problems and SLB problems arising in support vector machines. The methods were terminated once $\|x_k-x_{k-1}\|_2\leq\epsilon$ or the total number of iterations exceeds $100000$. We set $\tau_1=0.5$ and $\gamma=1.3$ for Algorithm \ref{al1} and used default parameters for the Dai-Fletcher and EQ-VABBmin methods.

\subsubsection{Random SLB problems}
We employ the procedure in \cite{dai2006new} to generate random SLB problems, which is based on the generation
of random SPD box-constrained quadratic test problems in \cite{More1989}. In particular, it uses five parameters $n,~ncond,~ndeg,~na(x^*)$ and $na(x_1)$ to determine the number of variables, condition number of the Hessian, amount of near-degeneracy, active variables at the solution $x^*$ and active variables at the starting point $x_1$, respectively. It generates SLB problems in the following form
\begin{equation*}
\begin{aligned}[t]
\min~~&\frac{1}{2}x \tr Ax-c \tr x\\
s.t.~~&l\leq x\leq u\\
~~&a \tr x=b,
\end{aligned}
\end{equation*}
where $A=PDP \tr $ with
$P=(I-2p_3p_3 \tr )(I-2p_2p_2 \tr )(I-2p_1p_1 \tr ),$
$D$ is a diagonal matrix whose $i$-th component is defined by
$$\log d_i=\frac{i-1}{n-1}ncond,~i=1,\ldots,n,$$
and $p_i$, $i=1,2,3,$ are random vectors whose elements are sampled from a uniform distribution in $(-1,1)$. See \cite{dai2006new} for more details on the generation of the problems.

\begin{figure}[thp!b]
	\centering
	\includegraphics[width=0.6\textwidth,height=0.45\textwidth]{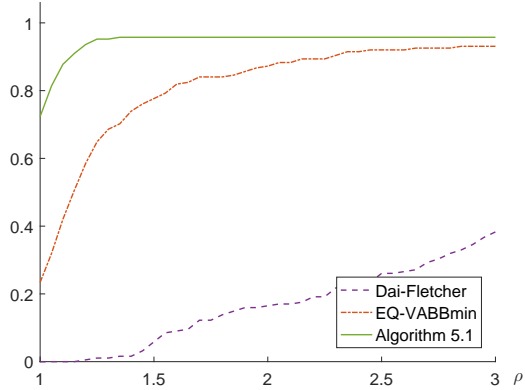}\\
	\caption{Performance profiles of Algorithm \ref{al1}, Dai-Fletcher and EQ-VABBmin on random SLB problems, iteration metric.}
	\label{spd}
\end{figure}

In our test, we set $n=10000$ and chose other parameters from
$$ncond\in\{4,5,6\},~ndeg\in\{1,3,5,7,9\},~na(x^*)\in\{1000,5000,9000\}.$$
We randomly generated one problem for each case and then got $45$ problems in total. Four different starting points with $na(x_1)\in\{0,1000,5000,9000\}$ were generated for each problem. The tolerance parameter $\epsilon$ was set to $10^{-8}$ for this test.

Figure \ref{spd} plots the performance profiles of the three algorithms based on the required number of iterations.
From Figure \ref{spd}, we can see that the EQ-VABBmin method is better than the Dai-Fletcher method, whereas
Algorithm \ref{al1} performs much better than both methods.

\subsubsection{Support vector machines}
Support vector machines (SVMs) are popular models in machine learning, especially suitable for classification, which can be formulated as an SLB problem, see \cite{cortes1995support} for details.  Suppose that we are given a training set of labeled examples
\begin{equation*}
D=\{(z_i,w_i),~i=1,\ldots,n,~z_i\in\mathbb{R}^m,~w_i\in\{1,-1\}\}.
\end{equation*}
The SVM model employs some kernel function, say $K:\mathbb{R}^m\times\mathbb{R}^m\rightarrow\mathbb{R}$, to classify new examples $z\in\mathbb{R}^m$ by a decision function $F:\mathbb{R}^m\rightarrow\{1,-1\}$ defined as
\begin{equation*}
F(z)=\mathrm{sign}\left(\sum_{i=1}^{n}x^*w_iK(z_i,z_j)+b^*\right),
\end{equation*}
where $x^*$ solves
\begin{equation*}
\begin{aligned}[t]
\min~~&\frac{1}{2}x \tr Gx-e \tr x\\
s.t.~~&0\leq x\leq Ce\\
~~&w \tr x=0.
\end{aligned}
\end{equation*}
Here, $G$ has entries $G_{ij}=w_iw_jK(z_i,z_j)$, $i,j=1,\ldots,n$, $C$ is a parameter of
the SVM model, and $e$ is the vector of all ones. The quantity $b^*\in\mathbb{R}$ is easily derived after $x^*$ is computed.

Using the Gaussian kernel
\begin{equation*}
K(z_i,z_j)=\exp\left(-\frac{\|z_i-z_j\|_2^2}{2\sigma^2}\right),
\end{equation*}
we have tested three real-world datasets for the binary classification: a9a, w8a, and ijcnn1, which can be downloaded from the LIBSVM website\footnote{\url{www.csie.ntu.edu.tw/~cjlin/libsvmtools/}}. For each dataset, we randomly chose $1000$ examples to generate the test problem. The parameters $C$ and $\sigma$ were set to $1$ and $\sqrt{10},$ respectively.

\begin{table}[htp!b]
	\setlength{\tabcolsep}{1.2ex}
	\caption{Results of Algorithm \ref{al1}, Dai-Fletcher and EQ-VABBmin on SVM problems with $n=1000$.}\label{tbsvm}
	\centering
	\begin{tabular}{|c|c|c|c|c|c|c|}
		\hline
		\multicolumn{1}{|c|}{\multirow{2}{*}{methods}} &\multicolumn{2}{c|}{\multirow{1}{*}{$10^{-3}$}}
		&\multicolumn{2}{c|}{\multirow{1}{*}{$10^{-6}$}}&\multicolumn{2}{c|}{\multirow{1}{*}{$10^{-9}$}}\\
		\cline{2-7}
		&   iter  &  CPU   &   iter  &  CPU   &   iter  &  CPU  \\
		\hline			
		&\multicolumn{6}{c|}{\multirow{1}{*}{a9a}}\\
		\cline{2-7}	
		Dai-Fletcher &   253  &  0.70  &   537  &  1.61  &   687  &  2.22  \\
		EQ-VABBmin   &   132  &  0.46  &   300  &  0.86  &   453  &  1.30  \\
		Algorithm \ref{al1}    &   121  &  0.38  &   253  &  0.84  &   468  &  1.54  \\
		\hline
		&\multicolumn{6}{c|}{\multirow{1}{*}{w8a}}\\
		\cline{2-7}
		Dai-Fletcher &   202  &  0.55  &   782  &  2.21  &  1028  &  2.73 \\
		EQ-VABBmin &   107  &  0.32  &   454  &  1.33  &   679  &  1.92 \\
		Algorithm \ref{al1}  &    78  &  0.24  &   284  &  0.81  &   550  &  1.59 \\
		
		\hline
		&\multicolumn{6}{c|}{\multirow{1}{*}{ijcnn1}}\\
		\cline{2-7}
		Dai-Fletcher &   293  &  0.81  & 41784  &145.69  & 84698  &290.48 \\
		EQ-VABBmin  &   123  &  0.39  & 23928  & 82.99  & 29647  & 90.57 \\
		Algorithm \ref{al1}  &    63  &  0.17  &  6487  & 20.44  & 21238  & 65.22 \\
		\hline
	\end{tabular}
\end{table}
Three different tolerance values were tested: $\epsilon=10^{-3},10^{-6},10^{-9}$. The null vector was employed as the initial point for all the compared methods. Table \ref{tbsvm} presents the required number of iterations and CPU time in seconds costed by the compared methods for different tolerance requirements. It can be seen from Table \ref{tbsvm} that Algorithm \ref{al1} often performs better than the other two methods.

\section{Conclusion and discussion}\label{secclu}

We have introduced a mechanism for the gradient method to achieve the two dimensional quadratic termination. Based on the mechanism, we derived a novel stepsize $\alpha_k^{new}$ (see the formula \eqref{snewa}) such that the Barzilai-Borwein (BB) method enjoys the two dimensional quadratic termination by equipping with $\alpha_k^{new}$. This novel stepsize only makes use of the BB stepsizes in previous iterations and thus can easily be adopted to general unconstrained and constrained optimization. We developed a new efficient gradient method (see the method
\eqref{snbbmadp}) that adaptively takes long BB steps and short steps associated with $\alpha_k^{new}$ for unconstrained quadratic optimization. Then based on the method \eqref{snbbmadp} and two nonmonotone line searches, we were able to design efficient gradient algorithms for unconstrained optimization problems and extreme eigenvalues problems, see Algorithms \ref{alunc} and \ref{aleig}. By incorporating the gradient projection technique and taking the constraints into consideration, we developed an efficient projected gradient algorithm, Algorithm \ref{al1}, for both box-constrained optimization and singly linearly box-constrained (SLB) optimization problems. Our numerical experiments demonstrated the efficiency of these algorithms over many successful numerical methods in the literature.

The results achieved in this paper further emphasize the importance of the two dimensional quadratic optimization model in the analysis of gradient methods for optimization. We are wondering whether higher-dimensional quadratic
optimization models will be helpful in the construction of more efficient gradient methods. This is an interesting
issue worthwhile investigation.

\end{document}